\documentclass[a4paper, 12pt]{article}

\usepackage[british]{babel}

\usepackage{graphicx}

\usepackage{amsmath}
\usepackage{amssymb}
\usepackage{enumerate}
\usepackage{amsthm}
\usepackage{bbm} 

\usepackage{mathrsfs} 

\makeatletter
\def\@swap#1#2{\let\@tempa#1\let#1#2\let#2\@tempa}
\@swap\epsilon\varepsilon
\makeatother

\newcommand{\RR}{\mathbb{R}}

\newcommand{\NN}{\mathbb{N}}

\newcommand{\Pro}{\mathbb{P}}
\newcommand{\FF}{\mathscr{F}}
\newcommand{\GG}{\mathscr{G}}
\newcommand{\Ind}{\mathbbm{1}}

\newcommand{\leftsub}[2]{{\protect\vphantom{#2}}_{#1}{#2}} 

\newcommand{\abs}[1]{\left\lvert #1 \right\rvert}

\newcommand{\stproca}[1]{\left(#1\right)_{t \ge 0}}
\newcommand{\stproc}[1]{\stproca{#1_t}}

\newcommand{\FFt}{\stproc{\FF}}

\DeclareMathOperator{\rRe}{Re}

\renewcommand{\Re}{\rRe}

\newcommand{\iu}{\mathrm{i}} 

\newcommand{\dint}{\displaystyle\int}

\usepackage[retainorgcmds]{IEEEtrantools}
\newenvironment{eqnarr}{\begin{IEEEeqnarray}{rCl}}{\end{IEEEeqnarray}\ignorespacesafterend}
\newenvironment{eqnarr*}{\begin{IEEEeqnarray*}{rCl}}{\end{IEEEeqnarray*}\ignorespacesafterend}
\newcommand{\eqnarrLHS}[1]{\IEEEeqnarraymulticol{3}{l}{#1} \\ \quad}
\newcommand{\eqnarrLHSnn}[1]{\IEEEeqnarraymulticol{3}{l}{#1} \nonumber \\ \quad}

\usepackage[final,bookmarksopen=true]{hyperref}

\newcommand{\xiLS}{\xi^{\mathrm{L}}}
\newcommand{\xiCPP}{\xi^{\mathrm{C}}}
\newcommand{\CE}{\Psi}
\newcommand{\CELS}{\CE^{\mathrm{L}}}
\newcommand{\CECPP}{\CE^{\mathrm{C}}}

\newcommand{\LD}{\pi}
\newcommand{\LDCPP}{\pi^{\mathrm{C}}}
\newcommand{\LDLS}{\pi^{\mathrm{L}}}

\newcommand{\LSabs}{\xi^*}

\newcommand{\stP}{\mathrm{P}}
\newcommand{\stPhat}{\mathrm{\hat P}}
\newcommand{\stE}{\mathrm{E}}
\newcommand{\stEhat}{\mathrm{\hat E}}

\newcommand{\stPcsx}[3]{\stP_{#1} \left[ #2 \middle\vert #3 \right]}

\newcommand{\stPup}{\stP^\uparrow}
\newcommand{\stPhatup}{\stPhat^\uparrow}
\newcommand{\Xup}{X^\uparrow}
\newcommand{\Xhatup}{\hat X^\uparrow}
\newcommand{\stEup}{\stE^\uparrow}

\newcommand{\LevP}{\mathbb{P}}
\newcommand{\LevPhat}{\mathbb{\hat P}}
\newcommand{\LevE}{\mathbb{E}} 
\newcommand{\LevEhat}{\mathbb{\hat E}}

\newcommand{\LevPhatup}{\LevPhat^\uparrow}
\newcommand{\xiup}{\xi^\uparrow}
\newcommand{\xihatup}{\hat \xi^\uparrow}
\newcommand{\LevEhatup}{\LevEhat^\uparrow}

\newcommand{\rhohat}{\hat{\rho}}

\newcommand{\jump}{\Delta}

\newcommand{\Ttrans}{\mathcal{T}}

\newcommand{\EsscherT}{\mathcal{E}}

\newcommand{\Ghgsymb}{{\leftsub{2}{\mathcal{F}}}_1}
\newcommand{\Ghg}[4]{\Ghgsymb(#1,#2;#3;#4)}

\newcommand{\Indic}[1]{\Ind_{(#1)}}
\newcommand{\Ych}{\check{Y}}

\newcommand{\GGt}{\stproca{\GG_t}}

\newcommand{\dd}{\mathrm{d}}

\newcommand{\fromto}[2]{#1\nobreakdash--#2}

\usepackage[retainorgcmds]{IEEEtrantools}

\newcommand{\LSdrift}{\mathtt{d}}

\newcommand{\pref}[1]{(\ref{#1})}

\newcommand{\taull}{{\tau \! -}}

\newcommand{\stparamset}{\mathcal{A}}
\newcommand{\upto}{\uparrow}
\DeclareMathOperator{\sgn}{sgn}

\theoremstyle{plain}
  \newtheorem{dummy}{***}[section]
  \newtheorem{thm}[dummy]{Theorem}
  \newtheorem{prop}[dummy]{Proposition}
  \newtheorem{cor}[dummy]{Corollary}
  \newtheorem{lem}[dummy]{Lemma}
\theoremstyle{definition}

  \newtheorem*{proof1}{Proof of Theorem \ref{interval hitting}, $\alpha \in (0,1]$}
  \newtheorem*{proof2}{Proof of Theorem \ref{interval hitting}, $\alpha \in (1,2)$}
  
  \newtheorem*{proof4}{Proof of Theorem \ref{HP 0 before up}}
  \newtheorem*{proof5}{Proof of Theorem \ref{potential int}}
  \newtheorem*{proofc1}{Proof of Corollary \ref{interval hitting prob}}
  \newtheorem*{proofp1}{Proof of Proposition \ref{Port limit}}
\theoremstyle{remark}
  \newtheorem{rem}[dummy]{Remark}

\newcommand{\define}{\emph} 

\newcommand{\for}{\qquad} 

\usepackage[numbers]{natbib}
\usepackage{doi}

\title{Hitting distributions of $\alpha$-stable processes via path censoring and self-similarity}
\author{{\large A. E. Kyprianou\footnote{{\sc University of Bath, UK.} E-mail: a.kyprianou@bath.ac.uk, aw295@bath.ac.uk}
\ \  J.C. Pardo\footnote{{\sc  CIMAT, Mexico.} E-mail: jcpardo@cimat.mx}
\ \ and
\  A. Watson$^*$
}
}

\begin{document}

\maketitle
\vspace{-1cm}

\begin{abstract}
\noindent
We consider two first passage problems for stable processes, not
necessarily symmetric, in one dimension.
We make use of a novel method of path censoring in order to
deduce explicit formulas for hitting probabilities, hitting distributions,
and a killed potential measure. To do this, we describe
in full detail the Wiener-Hopf factorisation
of a new Lamperti-stable-type L\'evy process obtained
via the Lamperti transform,
in the style of recent work in this area.


\end{abstract}

\vspace{0.5cm}

\section{Introduction}

A L\'evy process is a stochastic
process issued from the origin with stationary and
independent increments and  c\`adl\`ag paths.
If $X: = (X_t)_{t\geq 0}$ is a one-dimensional L\'evy process
with law $\stP$, then the classical L\'evy-Khintchine formula
states that for all $t\geq 0$ and $\theta\in \RR$,
the characteristic exponent
$\Psi(\theta) : = -t^{-1}\log \stE (e^{\iu\theta X_t})$ satisfies 
\[
\Psi(\theta) = \iu a\theta + \frac{1}{2}\sigma^2\theta^2 + \int_{\RR} (1 - e^{\iu\theta x} + \iu\theta x\Indic{|x|\leq 1})\Pi(\dd x),
\]
where $a\in\mathbb{R}$, $\sigma\geq 0$ and $\Pi$ is a measure
(the \define{L\'evy measure}) concentrated on
$\RR\setminus\{0\}$
such that $\int_{\RR}(1\wedge x^2)\Pi(\dd x)<\infty$.

$(X,\stP)$ is
said to be a \define{(strictly) $\alpha$-stable process}
if it is a L\'evy process which
also satisfies the \define{scaling property}: under $\stP$,
for every $c > 0$,
the process $(cX_{t c^{-\alpha}})_{t \ge 0}$
has the same law as $X$.
It is known that $\alpha \in (0,2]$, and the case $\alpha = 2$
corresponds to Brownian motion, which we exclude.
The L\'evy-Khintchine representation of such a process
is as follows:
$\sigma = 0$, and
$\Pi$ is absolutely continuous with density given by 
\[ 
c_+ x^{-(\alpha+1)} \Indic{x > 0} + c_- \abs{x}^{-(\alpha+1)} \Indic{x < 0},
  \for x \in \RR,
\]
where $c_+, c_- \ge 0$, and $c_+ = c_-$
when $\alpha = 1$.
It holds that $a = (c_+-c_-)/(\alpha-1)$ when $\alpha \ne 1$, and we
specify that $a = 0$ when
$\alpha = 1$; the latter condition is a restriction
which ensures that $X$ is a symmetric
process when $\alpha = 1$, so the only $1$-stable process we consider
is the symmetric Cauchy process.

These choices mean that,
up to a multiplicative constant $c>0$, $X$ has the canonical
characteristic exponent
\[
\Psi(\theta) = \begin{cases}
c\abs{\theta}^\alpha (1  - \iu\beta\tan\frac{\pi\alpha}{2}\sgn\theta) & \alpha\in(0,2)\setminus \{1\}, \\
c\abs{\theta} & \alpha = 1,
\end{cases}
\for\theta\in\mathbb{R},
\]
where $\beta = (c_+- c_-)/(c_+ + c_-)$. For more details, see \citet[\S 14]{Sato}.

For consistency with the literature we appeal to in this article, we shall always parameterise our $\alpha$-stable process such that 
\[ c_+ = \frac{\Gamma(\alpha+1)}{\Gamma(\alpha\rho)\Gamma(1-\alpha\rho)}
  \quad \text{and} \quad
  c_- = \frac{\Gamma(\alpha+1)}{\Gamma(\alpha\rhohat)\Gamma(1-\alpha\rhohat)},
  \]
where
$\rho = \stP(X_t \ge 0) = \stP(X_t > 0)$ is the positivity parameter, 
and $\rhohat = 1-\rho$. 

We take the point of view that the class of stable processes,
with this normalisation, is
parameterised by $\alpha$ and $\rho$; the reader will note that
all the quantities above can be written in terms of these parameters.
We shall restrict ourselves a little further within this class
by excluding the possibility of having only one-sided jumps.
Together with our assumption about the case $\alpha = 1$, this gives
us the following set of admissible parameters:
\begin{eqnarr*}
  \stparamset &=&
  \bigl\{ (\alpha,\rho) : \alpha \in (0,1), \, \rho \in (0,1) \bigr\} \\
  && {} \cup
  \bigl\{ (\alpha,\rho) : \alpha \in (1,2), \, \rho \in (1-1/\alpha, 1/\alpha) \bigr\}
  \cup
  \bigl\{ (\alpha, \rho) = (1, 1/2) \bigr\}.
\end{eqnarr*}

\medskip

After Brownian motion, $\alpha$-stable processes are often considered an exemplary family of processes for which many aspects of the general theory of L\'evy processes  can be illustrated in closed form.
First passage problems, which are relatively straightforward to handle in   the case of Brownian motion, become much harder in the setting of a general L\'evy process on account of the inclusion of jumps.  A collection of articles through the 1960s and early 1970s, appealing largely to potential analytic methods for general Markov processes, were relatively successful in handling a number of first passage problems, in particular for symmetric $\alpha$-stable processes in one or more dimensions.
See, for example, \cite{BGR-st-hit, Get1, Get2, Por-htpr, Rog71} to name but a few.

However, following this cluster of activity, several decades have passed since new results on these problems have appeared. The last few years have seen a number of new, explicit first passage identities for one-dimensional $\alpha$-stable processes, thanks to a better understanding of the intimate relationship between the aforesaid processes and positive self-similar Markov processes. See, for example, \cite{CC06, CPP-bernoulli, ChaKypPar09, JCKuz, n-tuple}.

In this paper we return to the work of \citet{BGR-st-hit}, published in 1961, which gave the law of the position of first entry of a symmetric $\alpha$-stable process into the unit ball. Specifically, we are interested in establishing the same law, but now for all the one-dimensional $\alpha$-stable processes which fall within the parameter regime $\stparamset$;
we remark that \citet[\S3.1, Remark 3]{Por-htpr} found this law for
processes with one-sided jumps, which justifies our exclusion of
these processes in this work.
Our method is modern in the sense that we appeal to the relationship of $\alpha$-stable processes with certain positive self-similar Markov processes. However, there are two notable additional innovations. First, we make use of a type of path censoring. Second, we are able to describe in explicit analytical detail a non-trivial Wiener-Hopf factorisation of an auxiliary L\'evy process from which the desired solution can be sourced.  Moreover, as a consequence of this approach, we are able to deliver a number of additional, related identities in explicit form for $\alpha$-stable processes.

\medskip \noindent
We now state the main results of the paper.
Let $\stP_x$ refer to the law
of $X+x$ under $\stP$, for each $x\in\mathbb{R}$.
We introduce the first hitting time of the interval
$(-1,1)$,
\[ \tau_{-1}^1 = \inf\{ t > 0 : X_t \in (-1,1) \} .
\]
Note that, for $x \ne -1,1$,
$X_{\tau_{-1}^1} \in (-1,1)$ $\stP_x$-a.s.\ so long as $X$ is
not spectrally one-sided. However, in Proposition \ref{Port limit},
we will consider a spectrally negative $\alpha$-stable process,
for which $X_{\tau_{-1}^1}$ may take the value $-1$ with
positive probability.

\medskip
\begin{thm}\label{interval hitting}
Let $x > 1$. Then, when $\alpha \in (0,1]$,
\[ \stP_x(X_{\tau_{-1}^1} \in \dd y, \, \tau_{-1}^1 < \infty)/\dd y
  = \frac{\sin(\pi\alpha\rhohat)}{\pi}
    (x+1)^{\alpha\rho}
    (x-1)^{\alpha\rhohat}
    (1+y)^{-\alpha\rho}
    (1-y)^{-\alpha\rhohat}
    (x-y)^{-1}
    , \]
for $y \in (-1,1)$. When $\alpha \in (1,2)$,
\begin{eqnarr*}
 \stP_x(X_{\tau_{-1}^1} \in \dd y)/\dd y
 &=& \frac{\sin(\pi\alpha\rhohat)}{\pi}
    (x+1)^{\alpha\rho}
    (x-1)^{\alpha\rhohat}
    (1+y)^{-\alpha\rho}
    (1-y)^{-\alpha\rhohat}
    (x-y)^{-1} \\
  && {} - (\alpha-1)
    \frac{\sin(\pi\alpha\rhohat)}{\pi}
    (1+y)^{-\alpha\rho}
    (1-y)^{-\alpha\rhohat}
    \int_1^x (t-1)^{\alpha\rhohat-1} (t+1)^{\alpha\rho-1}\, \dd t,
\end{eqnarr*}
for $y \in (-1,1)$.
\end{thm}

\noindent
When $X$ is symmetric, Theorem \ref{interval hitting}
reduces immediately to
Theorems B and C of \cite{BGR-st-hit}.
Moreover, the following hitting probability can be
obtained.

\begin{cor}\label{interval hitting prob}
When $\alpha \in (0,1)$, for $x > 1$,
\[ \stP_x( \tau_{-1}^1 = \infty)
  = \frac{\Gamma(1-\alpha\rho)}{\Gamma(\alpha\rhohat)\Gamma(1-\alpha)}
  \int_0^{\frac{x-1}{x+1}} t^{\alpha\rhohat - 1} (1-t)^{-\alpha} \, \dd t . \]
\end{cor}

\noindent
This extends Corollary 2 of \cite{BGR-st-hit}, as can be seen by
differentiating and using the doubling formula
\cite[8.335.2]{GR} for the gamma function.

The spectrally one-sided case can be found as the limit of
Theorem \ref{interval hitting},
as we now explain. The first part of the coming proposition is due to
\citet{Por-htpr}, but we re-state it for the sake of clarity.
\begin{prop}
\label{Port limit}
Let $\alpha \in (1,2)$, and suppose that $X$ is spectrally negative,
that is, $\rho = 1/\alpha$.
Then, the hitting distribution of $[-1,1]$ is given by
\begin{eqnarr*}
  \stP_x(X_{\tau_{-1}^{1}} \in \dd y)
  &=& \frac{\sin \pi(\alpha-1)}{\pi}
  (x-1)^{\alpha-1} (1-y)^{1-\alpha} (x-y)^{-1} \dd y \\
  && {} + \frac{\sin \pi(\alpha-1)}{\pi}
  \int_0^{\frac{x-1}{x+1}} t^{\alpha-2} (1-t)^{1-\alpha} \, \dd t
  \, \delta_{-1}(\dd y),
  \for x > 1, \, y \in [-1,1],
\end{eqnarr*}
where $\delta_{-1}$ is the unit
point mass at $-1$.
Furthermore, the measures on $[-1,1]$ given in Theorem
\ref{interval hitting} converge weakly, as
$\rho \to 1/\alpha$, to the limit above.
\end{prop}

The following killed potential is also available.

\begin{thm}\label{potential int}
Let $\alpha \in (0,1]$, $x > 1$ and $y>1$. Then,
\begin{eqnarr*}
  \eqnarrLHS{\stE_x \int_0^{\tau_{-1}^1} \Indic{X_t \in \dd y} \, \dd t / \dd y}
  &=&
  \begin{cases}
    \dfrac{1}{\Gamma(\alpha\rho)\Gamma(\alpha\rhohat)}
    \biggl(\dfrac{x-y}{2}\biggr)^{\alpha-1}
    \displaystyle\int_1^{\frac{1-xy}{y-x}}
      (t-1)^{\alpha\rho-1} (t+1)^{\alpha\rhohat-1} \, \dd t,
    & 1 < y < x, \\[1em]
    \dfrac{1}{\Gamma(\alpha\rho)\Gamma(\alpha\rhohat)}
    \biggl(\dfrac{y-x}{2}\biggr)^{\alpha-1}
    \displaystyle\int_1^{\frac{1-xy}{x-y}}
      (t-1)^{\alpha\rhohat-1} (t+1)^{\alpha\rho-1} \, \dd t,
    & y > x.
  \end{cases}
\end{eqnarr*}
\end{thm}

\noindent
To obtain the potential of the previous theorem  for $x < -1$, and $y < -1$, one may easily appeal to duality. 
In the case that $x<-1$ and $y>1$, one notes that 
\begin{equation}
 \stE_x \int_0^{\tau_{-1}^1} \Indic{X_t \in \dd y} \, \dd t   = \stE_x \stE_\Delta \int_0^{\tau_{-1}^1} \Indic{X_t \in \dd y} \, \dd t , 
 \label{DELTA}
\end{equation}
where the quantity $\Delta$ is randomised
according to the distribution of $X_{\tau^+_{-1}}\Indic{X_{\tau^+_{-1}}>1}$, with
\[
\tau^+_{-1}= \inf \{ t > 0 : X_t > -1 \}.
\]
Although the distribution of $X_{\tau^+_{-1}}$ is
available from \cite{Rogozin72}, and hence the right hand side of
(\ref{DELTA}) can be written down explicitly, it does not seem
to be easy to find a convenient closed form expression for the
corresponding potential density.

Regarding this potential, let us finally remark that our methods
give an explicit expression for this potential even
when $\alpha \in (1,2)$, but again, there does not seem to be
a compact expression for the density.

\bigskip

\noindent
A further result concerns the first passage of $X$ into the half-line $(1,\infty)$  before hitting zero. Let
\[ 
\tau_1^+ = \inf \{ t > 0 : X_t > 1 \} \text{ and }
  \tau_0 = \inf \{ t > 0 : X_t = 0 \} . 
  \]
Recall that when $\alpha \in (0,1]$, $\stP_x(\tau_0 = \infty) = 1$,
while when $\alpha \in (1,2)$, $\stP_x(\tau_0 < \infty) = 1$,
for $x \ne 0$.
In the latter case, we can obtain a hitting probability as follows.

\medskip
\begin{thm}\label{HP 0 before up}
Let $\alpha \in (1,2)$. When $0 < x < 1$,
\[ \stP_x(\tau_0 < \tau_1^+)
  = (\alpha-1)
  x^{\alpha-1}
  \int_1^{1/x} (t-1)^{\alpha\rho-1} t^{\alpha\rhohat-1} \, \dd t .
\]
When $x < 0$, 
\[ \stP_x(\tau_0 < \tau_1^+)
  = (\alpha-1)
  (-x)^{\alpha-1}
  \int_1^{1-1/x} (t-1)^{\alpha\rhohat-1} t^{\alpha\rho-1}\, \dd t .
\]
\end{thm}

It is not difficult to push Theorem \ref{HP 0 before up} a little further to give the law of the position of first entry into $(1,\infty)$ on the event $\{\tau^+_1<\tau_0\}$. Indeed,  by the Markov property, for $x < 1$,
\begin{eqnarr}
  \stP_x(X_{\tau_1^+} \in \dd y, \, \tau_1^+ < \tau_0)
  &=& \stP_x(X_{\tau_1^+} \in \dd y) - \stP_x(X_{\tau_1^+} \in \dd y, \tau_0 < \tau_1^+) \notag\\
  &=& \stP_x(X_{\tau_1^+} \in \dd y) - \stP_x(\tau_0 < \tau_1^+) \stP_0(X_{\tau_1^+} \in \dd y).
 \label{put-in-Rog}
\end{eqnarr}
Moreover, \citet{Rogozin72} found that, for  $x < 1$ and $y>1$,
\begin{equation} \stP_x(X_{\tau_1^+} \in \dd y)
  = \frac{\sin(\pi\alpha\rho)}{\pi} (1-x)^{\alpha\rho} (y-1)^{-\alpha\rho} (y-x)^{-1} \, \dd y.
  \label{Rog-first}
   \end{equation}
Hence substituting (\ref{Rog-first}) together with the hitting probability from Theorem \ref{HP 0 before up} into (\ref{put-in-Rog})
yields the following corollary.
\medskip
\begin{cor}\label{FP before 0}
Let $\alpha \in (1,2)$ Then, when $0 < x < 1$,
\[
\begin{split}
   \stP_x(X_{\tau_1^+} \in \dd y, \, \tau_1^+ < \tau_0) / \dd u
  &= \frac{\sin(\pi\alpha\rho)}{\pi} (1-x)^{\alpha\rho} (y-1)^{-\alpha\rho} (y-x)^{-1} \\
  & {} - (\alpha-1)
    \frac{\sin(\pi\alpha\rho)}{\pi}
    x^{\alpha-1}
    (y-1)^{-\alpha\rho}
    y^{-1}
    \int_1^{1/x} (t-1)^{\alpha\rho-1} t^{\alpha\rhohat-1}\, \dd t,
\end{split}
\]
for $y>1.$ When $x < 0$,
\[
\begin{split}
  \stP_x(X_{\tau_1^+} \in \dd y, \, \tau_1^+ < \tau_0)& / \dd y
  = \frac{\sin(\pi\alpha\rho)}{\pi} (1-x)^{\alpha\rho} (y-1)^{-\alpha\rho} (y-x)^{-1} \\
  & {} - (\alpha-1)
    \frac{\sin(\pi\alpha\rho)}{\pi}
    (-x)^{\alpha-1}
    (y-1)^{-\alpha\rho}
    y^{-1}
    \int_1^{1-1/x} (t-1)^{\alpha\rhohat-1} t^{\alpha\rho-1}\, \dd t,
\end{split}
\]
for $y>1$.
\end{cor}

\noindent

\medskip \noindent
We conclude this section by giving an overview of the rest of the paper. In Section 2, we recall the Lamperti transform
and discuss its relation to $\alpha$-stable processes. In Section 3, we explain the operation
which gives us the path-censored $\alpha$-stable process $Y$, that is to say the $\alpha$-stable process with the negative components of its path removed.
We show that $Y$ is a positive self-similar Markov process, and can
therefore be written as the exponential of a time-changed L\'evy process, say $\xi$. We show that the L\'evy process 
$\xi$ can be decomposed into the sum of a compound Poisson process and a so-called
Lamperti-stable process. Section 4 is dedicated to finding the distribution of the
jumps of this compound Poisson component, which we then use in Section 5 to
compute in  explicit detail the Wiener-Hopf factorisation of $\xi$. Finally,
we make use of the explicit nature of the Wiener-Hopf factorisation in
Section 6 to prove Theorems \ref{interval hitting} and \ref{potential int}.
There we also prove Theorem \ref{HP 0 before up} via
a connection with the process conditioned to stay positive.


\section{Lamperti transform and Lamperti-stable processes}\label{LS procs}

A \define{positive self-similar Markov process} (\define{pssMp}) with
\define{self-similarity index} $\alpha > 0$ is a standard Markov process
$Y = (Y_t)_{t\geq 0}$ with filtration $\GGt$ and probability laws
$(\stP_x)_{x > 0}$, on $[0,\infty)$, which has $0$ as an absorbing state and
which satisfies the \define{scaling property}, that for every $x, c > 0$,
\begin{equation}
  \label{scaling prop}%
  \text{ the law of } (cY_{t c^{-\alpha}})_{t \ge 0}
  \text{ under } \stP_x \text{ is } \stP_{cx} \text{.}
\end{equation}
Here, we mean ``standard'' in the sense of \cite{BG-mppt},
which is to say, $\GGt$ is a complete, right-continuous filtration,
and $Y$ has c\`adl\`ag paths and is strong Markov
and quasi-left-continuous.

In the seminal paper \cite{LampertiT}, Lamperti describes a one to one correspondence
between pssMps and L\'evy processes, which we now outline.
It may be worth noting that we have presented a slightly
different definition of pssMp from Lamperti; for the connection, see
\cite[\S 0]{VA-Ito}.

Let
$S(t) = \int_0^t (Y_u)^{-\alpha}\, \dd u .$
This process is continuous and strictly increasing until $Y$ reaches zero.
Let $(T(s))_{s \ge 0}$ be its inverse, and define
\[ \xi_s = \log Y_{T(s)} \qquad s\geq 0. 
\]
Then $\xi : = (\xi_s)_{s\geq 0}$ is a L\'evy process started at $\log x$, possibly killed at an independent
exponential time; the law of the L\'evy process and the rate of killing do not depend
on the value of $x$. The real-valued process $\xi$ with probability laws
$(\LevP_y)_{y \in \RR}$ is called the
\define{L\'evy process associated to $Y$}, or the \define{Lamperti transform of $Y$}.

An equivalent definition of $S$ and $T$, in terms of $\xi$ instead
of $Y$, is given by taking
$T(s) = \int_0^s \exp(\alpha \xi_u)\, \dd u$
and $S$ as its inverse. Then,
\begin{equation}
  \label{Lamp repr}
  Y_t = \exp(\xi_{S(t)}) 
\end{equation}
for all $t\geq 0$, and this shows that the Lamperti transform is a bijection.

Let $T_0 = \inf\{ t > 0: Y_t = 0 \}$ be the first hitting time of the absorbing state zero. Then the large-time behaviour of $\xi$ can be described by the behaviour of $Y$ at $T_0$, as follows:
\begin{enumerate}[(i)]
  \item If $T_0 = \infty$ a.s., then $\xi$ is unkilled and either oscillates or drifts to $+ \infty$.
  \item If $T_0 < \infty$ and $Y_{T_0 -} = 0$ a.s., then $\xi$ is unkilled and drifts to $-\infty$.
  \item If $T_0 < \infty$ and $Y_{T_0 -} > 0$ a.s., then $\xi$ is killed.
\end{enumerate}
It is proved in \cite{LampertiT} that the events mentioned above satisfy a zero-one law independently of $x$, and so the three possibilites above are an exhaustive classification of pssMps.

\medskip 

Three concrete examples of positive self-similar Markov processes related to $\alpha$-stable processes  
are treated in \citet{CC06}. We  present here the simplest case, namely that of the
$\alpha$-stable process absorbed at zero.
To this end, let $X$ be the $\alpha$-stable process
as defined in the introduction, and let
\[ \tau_0^- = \inf\{ t > 0 : X_t \le 0 \} . \]
\label{LSabs defined}%
Denote by $\LSabs$ the Lamperti transform of the pssMp $\stproca{X_t \Indic{t < \tau_0^-}}$.
Then $\LSabs$ has L\'evy density
\begin{equation}
\label{LSabs density}
  c_+ \frac{e^x}{(e^x-1)^{\alpha+1}} \Indic{x > 0}
  + c_- \frac{e^x}{(1-e^x)^{\alpha+1}} \Indic{x < 0} ,
\end{equation}
 and is killed at rate
$c_-/\alpha = \frac{\Gamma(\alpha)}{\Gamma(\alpha\rhohat)\Gamma(1-\alpha\rhohat)}$.

We note here that in \cite{CC06} the authors assume
that $X$ is symmetric when $\alpha = 1$, which motivates
the same assumption in this paper.

\section{The censored process and its Lamperti transform}
\label{s:cens}

We now describe the construction of the censored $\alpha$-stable process that will lie at the heart of our analysis, show that it is a pssMp and discuss its Lamperti transform. 

Henceforth, $X$, with probability laws $(\stP_x)_{x \in \RR}$, will denote the $\alpha$-stable process defined
in the introduction. 
Define the
occupation time of $(0,\infty)$,
\[ A_t = \int_0^t \Indic{X_s > 0} \, \dd s , \]
and let $\gamma(t) = \inf\{ s \ge 0 : A_s > t \}$ be its right-continuous inverse.
Define a process $(\Ych_t)_{ t \ge 0}$ by setting $\Ych_t = X_{\gamma(t)}$, $t\geq 0$. This is the process
formed by erasing the negative components of $X$ and joining up the gaps.

\medskip
\noindent
Write $\FFt$ for the augmented natural filtration of $X$, and
$\GG_t = \FF_{\gamma(t)}$, $t \ge 0$.

\begin{prop}\label{Ych scaling}
The process $\Ych$ is strong Markov with respect to the filtration $\GGt$ and satisfies the
scaling property with self-similarity index $\alpha$.
\begin{proof}
The strong Markov property follows directly from \citet[III.21]{RW1}. Establishing the scaling property is a straightforward exercise.
\end{proof}
\end{prop}

\noindent
We now make zero into an absorbing state. Define the stopping time
\[ T_0 = \inf\{ t > 0 : \Ych_t = 0 \} \]
and the process
\[ Y_t = \Ych_t \Indic{t < T_0} , \for t \ge 0 , \]
so that $Y :=(Y_t)_{t\geq 0}$ is $\Ych$ absorbed at zero.
We call the process $Y$ with probability laws $(\stP_x)_{x > 0}$
the \define{path-censored $\alpha$-stable process}.

\begin{prop}
The process $Y$ is a pssMp with respect to the filtration $\GGt$.
\begin{proof}
The scaling property follows from Proposition \ref{Ych scaling},
and zero is evidently an absorbing state. It remains to show that
$Y$ is a standard process, and the only point which may be in doubt
here is quasi-left-continuity. This follows from the Feller
property, which in turn follows from scaling and the Feller
property of $X$.
\end{proof}
\end{prop}

\medskip
\begin{rem}\label{r:other pssmps}
The definition of $Y$ via time-change and stopping at zero
bears some resemblance to
a number of other constructions:
\begin{enumerate}[(a)]
\item Bertoin's
construction \cite[\S 3.1]{Ber-split} of the L\'evy process conditioned
to stay positive. The key difference here is that, when a negative excursion is
encountered, instead of simply erasing it, \cite{Ber-split} patches the last jump
from negative to positive onto the final value of the previous positive excursion.
\item Bogdan, Burdzy and Chen's ``censored stable process''
for the domain $D = (0,\infty)$;
see \cite{BBC-cens}, in particular Theorem 2.1 and the preceding discussion.
Here the authors suppress any jumps of a symmetric $\alpha$-stable process
$X$ by which the process
attempts to escape the domain, and kill the process if it reaches
the boundary continuously.
\end{enumerate}
Both processes (a) and (b) are also pssMps with index $\alpha$.
%
These processes, together with the process $Y$ just described, therefore
represent three choices of how to restart an $\alpha$-stable process
in a self-similar way after it leaves
the positive half-line. We illustrate this in Figure \ref{fig:cens-diag}.

\end{rem}

\begin{figure}[tbp]
	\centering
		\includegraphics{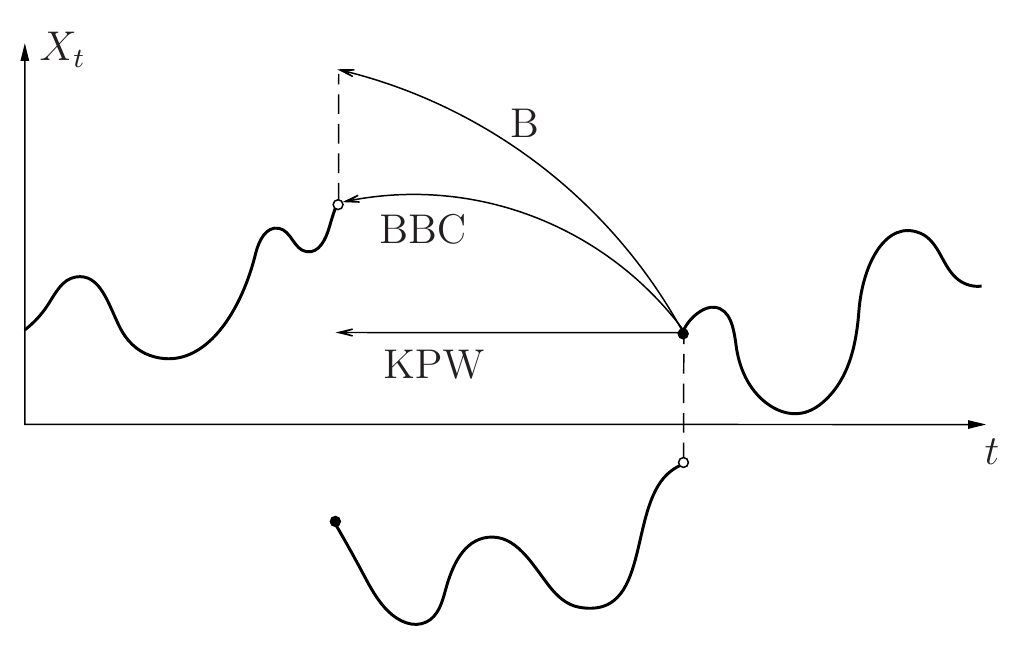}
	\caption{The construction of
	  three related processes from $X$, the stable process:
	  `B' is the stable process conditioned to stay positive \cite{Ber-split};
	  `BBC' is the censored stable process \cite{BBC-cens};
	  and `KPW' is the process $Y$ in this work.}
  \label{fig:cens-diag}
\end{figure}

\medskip\noindent
We now consider the pssMp $Y$ more closely for different
values of $\alpha\in(0,2)$. Taking account of
\citet[Proposition VIII.8]{BertoinLP} and the discussion immediately
above it we know that for $\alpha \in (0, 1]$, points are polar for $X$.
That is, $T_0 = \infty$ a.s.,
and so in this case $Y = \Ych$. Meanwhile, for $\alpha \in (1,2)$,
every point is recurrent,
so $T_0 < \infty$ a.s.. However, the process $X$ makes infinitely many jumps across zero before hitting it. Therefore, in this case $Y$ approaches zero continuously. In fact, it can be
shown that, in this case, $\Ych$ is the \define{recurrent extension} of $Y$ in the spirit of
\cite{RiveroRE1} and \cite{Fitz}.

Now, let $\xi  =(\xi_s)_{s\geq 0}$ be the Lamperti transform of $Y$. That is,
\begin{equation}
  \xi_s = \log Y_{T(s)} , \for s \ge 0, \label{e:LT of Y}
\end{equation} 
where $T$ is a time-change.
As in Section 2, we will write $\LevP_y$ for the law of $\xi$ started at
$y \in \RR$; note that $\LevP_y$ corresponds to $\stP_{\exp(y)}$.
The space transformation \eqref{e:LT of Y}, together with the above comments and,
for instance, the remark on p.~34 of \cite{BertoinLP}, allows us to make the following
distinction based on the value of $\alpha$.
\begin{enumerate}[(i)]
  \item If $\alpha \in (0,1)$, $T_0 = \infty$ and $X$ (and hence $Y$) is transient a.s..
    Therefore, $\xi$ is unkilled and drifts to $+ \infty$.
  \item If $\alpha = 1$, $T_0 = \infty$ and every neighbourhood of zero is an a.s.\ recurrent
    set for $X$, and hence also for $Y$.
    Therefore, $\xi$ is unkilled and oscillates.
  \item If $\alpha \in (1,2)$, $T_0 < \infty$ and $Y$ hits zero continuously.
    Therefore, $\xi$ is unkilled and drifts to $- \infty$.
\end{enumerate}

\noindent
Furthermore, we have the following result.

\begin{prop}
\label{xi repr}%
The L\'evy process $\xi$ is the sum of two independent L\'evy processes
$\xiLS$ and $\xiCPP$, which are characterised as follows:
\begin{enumerate}[(i)]
  \item\label{xi repr 1}%
    The L\'evy process $\xiLS$ has characteristic exponent
    \[ \Psi^*(\theta) - c_-/\alpha, \for \theta \in \RR , \]
    where $\Psi^*$ is the characteristic exponent of the process $\xi^*$ defined in
    Section \ref{LSabs defined}. That is, $\xiLS$ is formed by removing the independent
    killing from $\xi^*$.
    \item The process $\xiCPP$ is a compound Poisson process whose jumps occur
    at rate $c_-/\alpha$.
\end{enumerate}
\end{prop}

\noindent
Before beginning the proof, let us make some preparatory remarks.
Let
\[
  \tau = \inf\{ t > 0 : X_t < 0 \}
  \quad \text{and} \quad
  \sigma = \inf\{ t > \tau : X_t > 0 \}
\]
be hitting and return times of $(-\infty,0)$ and $(0,\infty)$ for $X$.
Note that, due to the time-change $\gamma$, $Y_\tau = X_\sigma$,
while $Y_{\tau -} = X_{\tau -}$. We require the following lemma.


\begin{lem} \label{Xtau law}%
The joint law of $(X_\tau,X_\taull,X_\sigma)$ under $\stP_x$ is equal to that
of $(x X_\tau, x X_\taull, x X_\sigma)$ under $\stP_1$.
\begin{proof}
This can be shown in a straightforward way using the scaling property.
\end{proof}
\end{lem}

\medskip
\begin{proof}[Proof of Proposition \ref{xi repr}]
First we note that, applying the strong Markov property to the
$\GGt$-stopping time $\tau$, it is sufficient
to study the process $(Y_t)_{t \le \tau}$.

It is clear that the path section $(Y_t)_{ t < \tau}$ agrees with
$(X_t)_{ t < \tau_0^-}$;
however, rather than being killed at time $\tau$, the
process $Y$ jumps to a positive state. Recall now that the effect of the Lamperti
transform on the time $\tau$ is to turn it into an
exponential time of rate $c_-/\alpha$ which is independent of $(\xi_s)_{s < S(\tau)}$.
This immediately yields the decomposition
of $\xi$ into the sum of
$\xiLS: = (\xiLS_s)_{s\geq 0}$ and $\xiCPP: = (\xiCPP_s)_{s\geq 0}$, 
where 
$\xiCPP$ is a process which jumps at the times of a Poisson
process with rate $c_-/\alpha$,
but whose jumps may depend on the position of $\xi$ prior to this jump.
What remains is to be  shown is that the values of the
jumps of $\xiCPP$ are also independent of $\xiLS$.

By the remark at the beginning of the proof, it is sufficient to show that the
first jump of $\xiCPP$ is independent of the previous path
of $\xiLS$.
Now, using only the independence of the jump times of $\xiLS$ and $\xiCPP$, we can
compute
\begin{eqnarr*}
  \jump Y_{\tau} := Y_{\tau} - Y_{\tau -}
  &=& \exp(\xiLS_{S(\tau)} + \xiCPP_{S(\tau)})
    - \exp(\xiLS_{S(\tau)-} + \xiCPP_{S(\tau) -}) \\
  &=& \exp(\xi_{S(\tau)-}) \bigl[ \exp(\jump \xiCPP_{S(\tau)}) - 1 \bigr] \\
  &=& X_{\taull} \bigl[ \exp(\jump \xiCPP_{S(\tau)}) - 1 \bigr] ,
\end{eqnarr*}
where $S$ is the Lamperti time change for $Y$, and
$\jump \xiCPP_s = \xiCPP_s - \xiCPP_{s-}$. Now,
\[
  \exp(\jump \xiCPP_{S(\tau)})= 1 + \frac{\jump Y_{\tau}}{X_{\taull}} = 1 + \frac{X_\sigma - X_{\taull}}{X_\taull} = \frac{X_\sigma}{X_\taull}.
\]
Hence, it is sufficient to show that $\frac{X_\sigma}{X_\taull}$ is independent of
$(X_t, t < \tau)$. The proof of this is essentially the same as that of part (iii) in Theorem 4 from \citet{CPR}, which we reproduce here for clarity.

First, observe that one consequence of Lemma \ref{Xtau law} is that,
for $g$ a Borel function and $x > 0$,
\[ \stE_x \biggl[ g\biggl( \frac{X_\sigma}{X_\taull} \biggr)\biggr]
  = \stE_1 \biggl[ g\biggl( \frac{X_\sigma}{X_\taull} \biggr) \biggr] . \]
Now, fix $n \in \NN$, $f$ and $g$ Borel functions and $s_1 < s_2 < \dotsb < s_n = t$.
Then, using the Markov property and the above equality,
\begin{eqnarr*}
  \stE_1 \biggl[ f(X_{s_1}, \dotsc, X_t)
    g\biggl( \frac{X_\sigma}{X_\taull} \biggr)
    \Indic{t < \tau} \biggr]
  &=& \stE_1 \biggl[
    f(X_{s_1}, \dotsc, X_t)
    \Indic{t < \tau}
    \stE_{X_t} \biggl[ g\biggl( \frac{X_\sigma}{X_\taull} \biggr) \biggr]
    \biggr] \\
  &=& \stE_1 \biggl[
    f(X_{s_1}, \dotsc, X_t)
    \Indic{t < \tau} \biggr]
    \stE_{1} \biggl[ g\biggl( \frac{X_\sigma}{X_\taull} \biggr) \biggr].
\end{eqnarr*}
We have now shown that $\xiLS$ and $\xiCPP$ are independent, and this completes the
proof.
\end{proof}

\begin{rem}
Let us consider the effect of the Lamperti transform on each of the
pssMps in Remark \ref{r:other pssmps}. For the process conditioned to stay
positive, the associated L\'evy process is
the process $\xiup$ defined in \citet{CC06}.
As regards the censored stable
process in $(0,\infty)$,
we can reason as in the above proposition to deduce that its Lamperti
transform is simply the process $\xiLS$ which we have just defined.
\end{rem}

\section{Jump distribution of the compound Poisson component}

In this section, we express the jump distribution of $\xiCPP$ in terms of known
quantitites, and hence derive its characteristic function and density.


Before stating a necessary lemma, we establish some more notation.
Let $\hat X$ be an independent copy of the dual process $-X$ and denote its probability laws by $(\stPhat_x)_{x \in \RR}$,
and let
\[ \hat\tau = \inf\{ t > 0 : \hat X_t < 0\} . \]
Furthermore, we shall denote by $\jump \xiCPP$ the random variable  whose law is the same as
the jump distribution of $\xiCPP$.

\medskip
\begin{lem}\label{jump decomp}
The random variable
$\exp(\jump \xiCPP)$ is equal in distribution to
\[  \biggl( - \frac{X_\tau}{X_{\taull}} \biggr)
  \Bigl( - \hat X_{\hat\tau} \Bigr)  , \]
where $X$ and $\hat X$ are taken to be independent with respective laws
$\stP_1$ and $\stPhat_1$.
\begin{proof}
In the proof of Proposition \ref{xi repr}, we saw that
\begin{equation}
  \exp(\jump \xiCPP_{S(\tau)})
  = \frac{X_\sigma}{X_\taull} . \label{decomp expr1}
\end{equation}
Applying the Markov property, and then using Lemma \ref{Xtau law} with the
$\alpha$-self-similar process $\hat X$, we obtain
\begin{eqnarr*}
  \stP_1(X_\sigma \in \cdot \vert \FF_\tau)
  &=& \stPhat_{-y}(-{\hat X}_{\hat \tau} \in \cdot)\big\vert_{y = X_\tau} \\
  &=& \stPhat_1(y{\hat X}_{\hat \tau} \in \cdot)\big\vert_{y = X_\tau}.
\end{eqnarr*}
Then, by disintegration,
\begin{eqnarr*}
  \stE_{1}\biggl[f\biggl(\frac{X_\sigma}{X_\taull}\biggr)\biggr]
  = \stE_{1}\biggl[
  \stE_{1}\biggl[
  f\biggl(\frac{X_\sigma}{X_\taull}\biggr)
  \bigg\vert\FF_\tau \biggr]
  \biggr]
  &=& \stE_1 \biggl[
    \int f\biggl(\frac{x}{X_\taull}\biggr)
    \stPcsx{1}{X_\sigma \in \dd x}{\FF_\tau}
    \biggr] \\
  &=& \stE_1 \biggl[
    \int f\biggl(\frac{x}{X_\taull}\biggr)
    \stPhat_1
    \bigl[y \hat X_{\hat\tau} \in \dd x \bigr]\big\vert_{y = X_{\tau}}
    \biggr] \\
  &=& \stE_1\biggl[
    \stEhat_1\biggl[ f\biggl(\frac{y\hat X_{\hat \tau}}{z}\biggr)
    \biggr]\bigg\vert_{y = X_{\tau}, \, z = X_{\taull}} \biggr] \\
  &=& \stE_1 \otimes \stEhat_1
  \biggl[ f\biggl(\frac{X_\tau \hat X_{\hat\tau}}{X_\taull}\biggr)
  \biggr] .
\end{eqnarr*}
Combining this with \eqref{decomp expr1}, we obtain that the law under $\stP_1$
of $\exp\bigl(\jump \xiCPP_{S(\tau)}\bigr)$ is equal to
that of $\dfrac{X_\tau \hat X_{\hat \tau}}{X_\taull}$ under $\stP_1 \otimes \stPhat_1$,
which establishes the claim.
\end{proof}
\end{lem}

\label{s:jump cf}
\noindent
The characteristic function of $\jump \xiCPP$ can now be found by rewriting the expression
in Proposition \ref{jump decomp} in terms of overshoots and undershoots of stable L\'evy
processes, whose marginal and joint laws are given in \citet{Rogozin72} and
\citet{DoKyp06}. Following the notation of \cite{DoKyp06}, let
\[ \tau_a^+ = \inf\{ t > 0 : X_t > a \} , \]
and let $\hat\tau_a^+$ be defined similarly for $\hat X$.

\bigskip
\begin{prop}
The characteristic function of the jump distribution of $\xiCPP$ is
given by
\begin{equation}
\label{jump cf}%
  \LevE_0 \bigl[ \exp\bigl(\iu\theta \jump \xiCPP\bigr) \bigr]
  = \frac{\sin(\pi \alpha \rho)}{\pi \Gamma(\alpha)}
    \Gamma(1-\alpha\rho + \iu\theta) \Gamma(\alpha\rho - \iu\theta)
    \Gamma(1 + \iu\theta) \Gamma(\alpha - \iu\theta).
\end{equation}

\begin{proof}
In the course of the coming computations, we will make use several times of the
beta integral,
\[ \int_0^1 s^{x-1} (1-s)^{y-1} \, \dd s
  = \int_0^\infty \frac{t^{x-1}}{(1+t)^{x+y}}\, \dd t
  = \frac{\Gamma(x) \Gamma(y)}{\Gamma(x+y)} \text{,} \for \Re x, \, \Re y > 0.
   \]
See for example \cite[formulas \fromto{8.830.1}{3}]{GR}.


Now, for $\theta \in \RR$,
\begin{equation}\label{jd 1}
\begin{split}
  \stEhat_1\biggl(-\hat X_{\hat\tau}\biggr)^{\iu\theta}&= \stE_0 \biggl( X_{\tau_1^+} - 1 \biggr)^{\iu\theta}  = \frac{\sin(\pi \alpha \rho)}{\pi}
    \int_0^\infty t^{\iu\theta - \alpha \rho} (1+t)^{-1}\, \dd t \\
    &= \frac{\sin(\pi \alpha \rho)}{\pi}
    \Gamma(1-\alpha \rho + \iu\theta) \Gamma(\alpha \rho - \iu\theta).
    \end{split}
\end{equation}
Furthermore,
\begin{equation}    \label{iteration}
\begin{split}
  \stE_1 &\biggl( \frac{X_\tau}{X_\taull} \biggr)^{\iu\theta}
  = \stEhat_0 \biggl(
      \frac{\hat X_{\hat\tau_1^+} - 1}{1 - \hat X_{\hat\tau_1^+ -}}
    \biggr)^{\iu\theta} \\
  &= \frac{\sin(\pi\alpha\rhohat)}{\pi}
  \frac{\Gamma(\alpha+1)}{\Gamma(\alpha\rho)\Gamma(\alpha\rhohat)}  \int_0^1 \int_y^\infty \int_0^\infty
    \frac{ u^{\iu\theta} (1-y)^{\alpha \hat{\rho} -1} (v-y)^{\alpha(1-\hat{\rho})-1} }
      { v^{\iu\theta} (v+u)^{1+\alpha} }
    \, \dd u \, \dd v \, \dd y.
\end{split}
\end{equation}
For the innermost integral above we have
\[
\int_0^\infty \frac{u^{\iu\theta}}{(u+v)^{1+\alpha}}\, \dd u
  \overset{w=u/v}{=}
  v^{\iu\theta - \alpha} \int_0^\infty \frac{w^{\iu\theta}}{(1+w)^{1+\alpha}}\, \dd w 
 =
v^{\iu\theta - \alpha} \frac{\Gamma(\iu\theta + 1)\Gamma(\alpha - \iu\theta)}{\Gamma(\alpha+1)} .
\]
The next iterated integral in \eqref{iteration} becomes
\[
\int_y^\infty v^{-\alpha} (v-y)^{\alpha(1-\rhohat) - 1}\, \dd v
  \overset{z=v/y-1}{=}
  y^{-\alpha\rhohat} \int_0^\infty \frac{z^{\alpha(1-\rhohat)-1}}{(1+z)^\alpha} \, \dd z 
  =
   y^{-\alpha\rhohat} \frac{\Gamma(\alpha\rho)\Gamma(\alpha\rhohat)}{\Gamma(\alpha)} ,
\]
and finally it remains to calculate
\[
  \int_0^1 y^{-\alpha\rhohat}(1-y)^{\alpha\rhohat-1} \, \dd y = \Gamma(1-\alpha\rhohat)\Gamma(\alpha\rhohat) . 
\]
Multiplying together these expressions 
and using the reflection identity
$\Gamma(x)\Gamma(1-x) = \pi/\sin(\pi x)$, we obtain
\begin{equation} 
  \stE_1 \biggl( - \frac{X_{\tau}}{X_{\tau \! -}} \biggr)^{\iu\theta}
    = \frac{\Gamma(\iu\theta + 1)\Gamma(\alpha- \iu\theta)}{\Gamma(\alpha)} .
  \label{jd 2}
\end{equation}

The result now follows from Lemma \ref{jump decomp} by multiplying \eqref{jd 1} and \eqref{jd 2} together.
\end{proof}
\end{prop}

\begin{rem}
The recent work of \citet{CPR} on the so-called Lamperti-Kiu processes can be applied to give
the same result. The quantity $\jump \xiCPP$ in the present work corresponds to the
independent sum $\xi^-_\zeta + U^+ + U^-$ in that paper, where $U^+$ and $U^-$ are
``log-Pareto'' random variables and $\xi^-$ is the Lamperti-stable process corresponding
to $\hat X$ absorbed below zero; see \cite[Corollary 11]{CPR}
for details. It is straightforward to show that the characteristic function of this sum is
equal to the right-hand side of \eqref{jump cf}.
\end{rem}

\noindent
It is now possible to deduce the density of the jump distribution from its
characteristic function.
By substituting on the left and using the beta integral, it can be shown that
\begin{eqnarr*}
  \int_{-\infty}^{\infty} e^{\iu \theta x}\,
    \alpha e^x (1+e^x)^{-(\alpha + 1)}\, \dd x
  &=& \frac{\Gamma(1+\iu\theta) \Gamma(\alpha - \iu\theta)}{\Gamma(\alpha)} , \\
  \int_{-\infty}^{\infty} e^{\iu \theta x}\,
    \frac{\sin(\pi\alpha\rho)}{\pi} e^{(1-\alpha\rho)x} (1+e^x)^{-1}\, \dd x
  &=& \frac{\sin(\pi\alpha\rho)}{\pi}
    \Gamma(\alpha\rho-\iu\theta) \Gamma(1-\alpha\rho+\iu\theta) ,
\end{eqnarr*}
and so the density of $\jump \xiCPP$ can be seen as the convolution of these two
functions. Moreover, it is even possible to calculate this convolution directly:
\begin{eqnarr}
\eqnarrLHSnn{
   \LevP_0\bigl(\jump\xiCPP \in \dd x\bigr)/\dd x
  = \frac{\alpha}{\Gamma(\alpha\rho)\Gamma(1-\alpha\rho)}
    \int_{-\infty}^{\infty} e^u
    (1+e^u)^{-(\alpha+1)}
    e^{(1-\alpha\rho)(x-u)}
    (1+e^{x-u})^{-1} \, \dd u}
  &=& \frac{\alpha}{\Gamma(\alpha\rho)\Gamma(1-\alpha\rho)} e^{-\alpha\rho x}
    \int_0^\infty t^{\alpha\rho} (1+t)^{-(\alpha+1)} (te^{-x} + 1)^{-1}\, \dd t \nonumber \\ 
  &=& \frac{\alpha}{\Gamma(\alpha\rho)\Gamma(1-\alpha\rho)}
    \frac{\Gamma(\alpha\rho+1) \Gamma(\alpha\rhohat + 1)}{\Gamma(\alpha+2)}
    e^{-\alpha\rho x} \Ghg{1}{\alpha\rho+1}{\alpha+2}{1-e^{-x}} ,
  \label{jump density}
\end{eqnarr}
where the final line follows from \cite[formula 3.197.5]{GR}, and is to be understood in the
sense of analytic continuation when $x < 0$.


\section{Wiener-Hopf factorisation}\label{WHF}%

We begin with a brief sketch of the Wiener-Hopf factorisation for L\'evy processes,
and refer the reader to
\citep[][Chapter 6]{Kyp} or \citep[][VI.2]{BertoinLP} for further details, including
proofs.

The Wiener-Hopf factorisation describes the characteristic exponent of a L\'evy process in terms of the Laplace exponents of two subordinators.
For our purposes, a \define{subordinator} is defined as an increasing L\'evy process,
possibly killed at an independent exponentially distributed time and sent to
the cemetary state $+\infty$. If $H$ is a
subordinator with expectation operator $\LevE$, we define its
\define{Laplace exponent} $\phi$ by the equation
\[ \LevE \bigl[ \exp(-\lambda H_1)\bigr] = \exp(-\phi(\lambda)), \for \lambda \ge 0 . \]
Standard theory allows us to analytically extend $\phi(\lambda)$ to $\{\lambda\in\mathbb{C}: \Re \lambda \geq 0\}$.
Similarly, let $\xi$ be a L\'evy process, again with expectation $\LevE$,
and denote its characteristic exponent
by $\CE$, so that
\[ \LevE\bigl[ \exp(\iu\theta \xi_1) \bigr] = \exp(-\CE(\theta)) ,
  \for \theta \in \RR. \]
The \define{Wiener-Hopf factorisation} of $\xi$ consists of the decomposition
\begin{equation}\label{the WHF}%
  k \CE(\theta) = \kappa(-\iu\theta) \hat \kappa( \iu\theta) , \for \theta \in \RR,
\end{equation}
where $k > 0$ is a constant which may, without loss of generality,
be taken equal to unity, and the functions $\kappa$ and $\hat \kappa$
are the Laplace exponents of certain subordinators which we denote $H$ and $\hat H$.

Any decomposition of the form \eqref{the WHF} is unique,
up to the constant $k$, provided that the
functions $\kappa$ and $\hat \kappa$ are Laplace exponents of subordinators.
The exponents $\kappa$ and $\hat \kappa$
are termed the \define{Wiener-Hopf factors} of $\xi$.

The subordinator $H$ can be identified in law as an appropriate time change of the running maximum process  $\bar \xi: = (\bar\xi_t)_{t\geq 0}$,
where $\bar \xi_t = \sup\{ \xi_s, \, s \le t\}$.
In particular, the range of $H$ and $\bar \xi$ are the same.
Similarly, $\hat H$ is equal in law to an appropriate time-change of $-\underline{\xi}: = (-\underline{\xi})_{t\geq 0}$, with
$\underline \xi_t = \inf\{\xi_s, \, s \le t \}$, and they have the same range.
Intuitively speaking, $H$ and $\hat H$ keep track of how $\xi$ reaches its new maxima
and minima, and they are therefore termed the 
\define{ascending} and \define{descending ladder height processes} associated
to $\xi$.


\medskip\noindent
In Sections \ref{s:small alpha WH} and \ref{s:big alpha WH}
we shall deduce in explicit form the Wiener-Hopf factors of $\xi$ from its characteristic exponent.
Analytically, we will need to distinguish the cases $\alpha \in (0,1]$ and 
$\alpha \in (1,2)$; in probabilistic terms, these correspond to the regimes where $X$
cannot and can hit zero, respectively.

Accordingly, the outline of this section is as follows. We first introduce two classes of
L\'evy processes and two transformations of subordinators which will be used to identify
the process $\xi$ and the ladder processes $H,\hat H$. We then present two subsections
with the same structure: first a theorem
identifying the factorisation and the ladder processes, and then a proposition
collecting some further details of important characteristics of the ladder height processes, which will be used in the applications.

\subsection{Hypergeometric L\'evy processes}\label{HGLPs}%
A process is said to be a \define{hypergeometric
L\'evy process} with parameters $(\beta,\gamma,\hat\beta,\hat\gamma)$ if it has
characteristic exponent
\[ \frac{\Gamma(1-\beta+\gamma-\iu\theta)}{\Gamma(1-\beta-\iu\theta)}
  \frac{\Gamma(\hat\beta+\hat\gamma+\iu\theta)}{\Gamma(\hat\beta + \iu\theta)}
  , \for \theta \in \RR \]
and the parameters lie in the admissible set
\[ \bigl\{ \beta \le 1, \, \gamma \in (0,1), \,
  \hat\beta \ge 0, \, \hat\gamma \in (0,1) \bigr\} . \]
In \citet{JCKuz} the authors derive the L\'evy measure and Wiener-Hopf factorisation
of such a process, and show that the processes $\xi^*$, $\xi^\uparrow$ and
$\xi^\downarrow$ of \citet{CC06} belong to this class;
these are, respectively, the L\'evy processes
appearing in the Lamperti transform of the $\alpha$-stable
process absorbed at zero, conditioned to stay positive and
conditioned to hit zero continuously.

\subsection{Lamperti-stable subordinators}\label{LSSs}%
A \define{Lamperti-stable subordinator} is characterised by parameters in the admissible set
\[ \{ (q, \mathtt{a}, \beta, c, \LSdrift) : \mathtt{a} \in (0,1),\ \beta\leq 1+\mathtt{a}, \, q, c, \LSdrift \ge 0 \} , \]
and it is defined as the (possibly killed) increasing L\'evy process with killing rate $q$, drift
$\LSdrift$, and L\'evy density
\[ c \frac{e^{\beta x}}{(e^x-1)^{\mathtt{a}+1}}, \for x > 0 . \]
It is simple to see from \cite[Theorem 3.1]{C-P-P}
that the Laplace exponent of such a process is given by
\begin{equation}\label{LSS LE}%
  \Phi(\lambda) = q+ \LSdrift \lambda - c \Gamma(-\mathtt{a})
    \left( \frac{\Gamma(\lambda + 1 - \beta + \mathtt{a})}{\Gamma(\lambda + 1 - \beta)}
      - \frac{\Gamma(1-\beta+\mathtt{a})}{\Gamma(1-\beta)} \right), \for \lambda \ge 0.
\end{equation}


\subsection{Esscher and \texorpdfstring{$\Ttrans_{\beta}$}{T-beta} transformations and special Bernstein functions}

The Lamperti-stable subordinators just introduced will not be sufficient to identify
the ladder processes associated to $\xi$ in the case $\alpha \in (1,2)$.
We therefore introduce two transformations of subordinators in order to expand
our repertoire of processes.

The first of these is the classical Esscher transformation,
a generalisation of the Cameron-Girsanov-Martin transformation
of Brownian motion. The second, the $\Ttrans_\beta$ transformation, is more recent,
but we will see that, in the cases we are concerned with, it is closely connected
to the Esscher transform. We refer the reader to \cite[\S 3.3]{Kyp} and
\cite[\S 2]{KP-CT} respectively for details.


The following result is classical.

\bigskip
\begin{lem}
Let $H$ be a subordinator with Laplace exponent $\phi$, and let $\beta > 0$.
Define the function
\[ \EsscherT_\beta \phi(\lambda) = \phi(\lambda + \beta) - \phi(\beta)
  , \for \lambda \ge 0 . \]
Then, $\EsscherT_\beta \phi$ is the Laplace exponent of a subordinator,
known as the \define{Esscher transform} of $H$ (or of $\phi$).

The Esscher transform of $H$ has no killing and the same drift coefficient as $H$,
and if the L\'evy measure of $H$ is $\Pi$, then its Esscher transform has
L\'evy measure $e^{-\beta x} \Pi(\dd x)$.
\end{lem}

\label{define sBf}
\noindent
Before giving the next theorem, we need to introduce the notions of special Bernstein
function and conjugate subordinators, first defined by \citet{SV-special}.
Consider a function
$\phi\colon [0,\infty) \to \RR$, and define
$\phi^* \colon [0,\infty) \to \RR$ by
\[ \phi^*(\lambda) = \lambda/\phi(\lambda) . \]
The function $\phi$ is called a \define{special Bernstein
function} if both $\phi$ and $\phi^*$ are the Laplace exponents of
subordinators. In this case, $\phi$ and $\phi^*$ are said
to be \define{conjugate} to one another, as are their corresponding subordinators.

\bigskip
\begin{prop}\label{t:Tbeta}%
Let $H$ be a subordinator with Laplace exponent $\phi$, and let $\beta > 0$. Define
\begin{equation}\label{eq:Tbeta}%
  \Ttrans_\beta \phi(\lambda)
  = \frac{\lambda}{\lambda+\beta} \phi(\lambda+\beta) , \for \lambda \ge 0.
\end{equation}
Then $\Ttrans_\beta \phi$ is the Laplace exponent of a subordinator with
no killing and the same drift coefficient as $H$.

Furthermore, if $\phi$ is a special Bernstein function
conjugate to $\phi^*$, then $\Ttrans_\beta \phi$ is a special
Bernstein function conjugate to
\[ \EsscherT_\beta \phi^* + \phi^*(\beta) . \]
\end{prop}

\begin{proof}
The first assertion is proved in \citet[p.~124]{Gne-rrcs} as the result of a path transformation,
and directly, for spectrally negative L\'evy processes (from which the
case of subordinators is easily extracted) in \citet{KP-CT}.
The killing rate and drift coefficient can be read off as $\Ttrans_\beta\phi(0)$ and
$\lim_{\lambda \to \infty} \Ttrans_\beta\phi(\lambda)/\lambda$.

The second claim can be seen immediately by rewriting \eqref{eq:Tbeta} as
\[ \Ttrans_\beta \phi(\lambda) = \frac{\lambda}{\phi^*(\lambda+\beta)} \]
and observing that
$\phi^*(\lambda+\beta) =  \EsscherT_\beta \phi^*(\lambda) + \phi^*(\beta)$
for $\lambda \ge 0$.
\end{proof}


\subsection{Wiener-Hopf factorisation for \texorpdfstring{$\alpha \in (0,1]$}{alpha in (0,1]}}
\label{s:small alpha WH}

\begin{thm}[Wiener-Hopf factorisation] \label{small alpha WH} \mbox{}
\begin{enumerate}[(i)] 
  \item \label{small alpha WH1}
    When $\alpha \in (0,1]$,
    the Wiener-Hopf factorisation of $\xi$ has components
    \begin{equation*}
      \kappa(\lambda) = \frac{\Gamma(\alpha\rho+\lambda)}{\Gamma(\lambda)},
      \qquad
      \hat\kappa(\lambda) = \frac{\Gamma(1-\alpha\rho+\lambda)}
        {\Gamma(1-\alpha+\lambda)} , \for \lambda \ge 0.
    \end{equation*}
    Hence, $\xi$ is a hypergeometric L\'evy process with parameters
    \[ \bigl(\beta,\gamma,\hat\beta,\hat\gamma\bigr)
      = \bigl(1, \alpha\rho, 1-\alpha, \alpha\rhohat \bigr). \]
  \item \label{small alpha WH2}
    The ascending ladder height process is a Lamperti-stable subordinator with
 parameters 
    \[ \bigl(q,  \mathtt{a}, \beta, c, \LSdrift \bigr)
      = \left(0, \alpha\rho, \, 1, \, -\frac{1}{\Gamma(-\alpha\rho)}, \, 0 \right). \]
  \item \label{small alpha WH3}
    The descending ladder height process is a Lamperti-stable subordinator with
    parameters
    \[ \bigl(q, \mathtt{a}, \beta, c, \LSdrift \bigr)
      = \left( \frac{\Gamma(1-\alpha\rho)}{\Gamma(1-\alpha)}, \alpha\rhohat, \, \alpha, \, -\frac{1}{\Gamma(-\alpha\rhohat)}, \, 0 \right) , \]
    when $\alpha < 1$, and 
    \[ \bigl(q, \mathtt{a}, \beta, c, \LSdrift \bigr)
      = \left( 0, \alpha\rhohat, \, \alpha, \, -\frac{1}{\Gamma(-\alpha\rhohat)}, \, 0 \right) , \]  
    when $\alpha = 1$.
\end{enumerate}

\begin{proof}
First we compute $\CECPP$ and $\CELS$, the characteristic exponents of $\xiCPP$ and
$\xiLS$. 
As $\CECPP$ is a compound Poisson process with jump rate $c_-/\alpha$ and jump distribution given by (\ref{jump cf}), it is immediate that, 
 for $\theta \in \RR$,
\[ \CECPP(\theta) = \frac{\Gamma(\alpha)}{\Gamma(\alpha\rhohat)\Gamma(1-\alpha\rhohat)}
  \left( 1 -
    \frac{\sin(\alpha \rho \pi)}{\pi \Gamma(\alpha)}
    \Gamma(1-\alpha\rho + \iu\theta) \Gamma(\alpha\rho - \iu\theta)
    \Gamma(1 + \iu\theta) \Gamma(\alpha - \iu\theta)
  \right) . \]
On the other hand, \cite[Theorem 1]{JCKuz} provides an expression for the characteristic exponent
$\Psi^*$ of the Lamperti-stable process $\xi^*$ from Section \ref{LS procs},
and removing the killing from this gives us
\[ \CELS(\theta)
  = \frac{\Gamma(\alpha-\iu\theta)}{\Gamma(\alpha\rhohat - \iu\theta)}
    \frac{\Gamma(1+\iu\theta)}{\Gamma(1-\alpha\rhohat + \iu\theta)} - 
    \frac{\Gamma(\alpha)}{\Gamma(\alpha\rhohat)\Gamma(1-\alpha\rhohat)} . \]
We can now compute
\begin{eqnarr*}
\CE(\theta) &=& \CELS(\theta) + \CECPP(\theta) \\
  &=& \Gamma(\alpha-\iu\theta)\Gamma(1+\iu\theta)
    \left(\frac{1}{\Gamma(\alpha\rhohat-\iu\theta)\Gamma(1-\alpha\rhohat+\iu\theta)} -
    \frac{\Gamma(1-\alpha\rho+\iu\theta)\Gamma(\alpha\rho-\iu\theta)}
      {\Gamma(\alpha\rho)\Gamma(1-\alpha\rho)\Gamma(\alpha\rhohat)\Gamma(1-\alpha\rhohat)}
    \right) \\
  &=& \Gamma(\alpha-\iu\theta)\Gamma(1+\iu\theta)
    \Gamma(1-\alpha\rho+\iu\theta)\Gamma(\alpha\rho-\iu\theta) \\
  && {} \times
    \left( \frac{\sin(\pi(\alpha\rhohat-\iu\theta)) \sin(\pi(\alpha\rho-\iu\theta))}{\pi^2}
    - \frac{\sin(\pi\alpha\rhohat) \sin(\pi\alpha\rho)}{\pi^2}
    \right).
\end{eqnarr*}
It may be proved, using product and sum identities for trigonometric functions, that
\[ \sin(\pi(\alpha\rhohat - \iu\theta)) \sin(\pi(\alpha\rho-\iu\theta))
  + \sin(\pi\iu\theta)\sin(\pi(\alpha - \iu\theta))
  = \sin(\pi\alpha\rhohat)\sin(\pi\alpha\rho).
\]
This leads to
\begin{eqnarr}
  \CE(\theta)
  &=& \Gamma(\alpha-\iu\theta)\Gamma(1+\iu\theta)
    \Gamma(1-\alpha\rho+\iu\theta)\Gamma(\alpha\rho-\iu\theta)
    \frac{\sin(-\pi\iu\theta)\sin(\pi(\alpha - \iu\theta))}{\pi^2} \nonumber \\
  &=& \frac{\Gamma(\alpha-\iu\theta)\Gamma(1+\iu\theta)}
      {\Gamma(1+\iu\theta)\Gamma(-\iu\theta)}
    \frac{\Gamma(\alpha\rho-\iu\theta)\Gamma(1-\alpha\rho+\iu\theta)}
      {\Gamma(\alpha-\iu\theta)\Gamma(1-\alpha+\iu\theta)} \nonumber \\
  &=& \frac{\Gamma(\alpha\rho - \iu\theta)}{\Gamma(-\iu\theta)}
    \times
    \frac{\Gamma(1 - \alpha\rho + \iu\theta)}{\Gamma(1 - \alpha + \iu\theta)} .
    \label{small alpha CE}
\end{eqnarr}
Part \pref{small alpha WH1} now follows by the uniqueness of the Wiener-Hopf
factorisation, once we have identified $\kappa$ and $\hat \kappa$ as Laplace exponents
of subordinators. Substituting the parameters in parts \pref{small alpha WH2}
and \pref{small alpha WH3} into the formula \eqref{LSS LE} for the Laplace
exponent of a Lamperti-stable subordinator, and adding killing in the case
of part \pref{small alpha WH3}, completes the proof.
\end{proof}
\end{thm}

\begin{prop}\label{small alpha p}%
\mbox{}
\begin{enumerate}[(i)]
  \item
    The process $\xi$ has L\'evy density
    \[ \LD(x) = 
      \begin{cases}
        - \dfrac{1}{\Gamma(1-\alpha\rhohat)\Gamma(-\alpha\rho)} e^{-\alpha\rho x}
          \Ghg{1+\alpha\rho}{1}{1-\alpha\rhohat}{e^{-x}},
        & \text{if}\ x > 0,  \\
        - \dfrac{1}{\Gamma(1-\alpha\rho)\Gamma(-\alpha\rhohat)} e^{(1-\alpha\rho) x}
          \Ghg{1+\alpha\rhohat}{1}{1-\alpha\rho}{e^{x}},
        & \text{if}\ x < 0.
      \end{cases}
    \]
  \item
    The ascending ladder height has L\'evy density
    \[ \pi_H(x) =
      - \frac{1}{\Gamma(-\alpha\rho)} e^x (e^x-1)^{-(\alpha\rho+1)}, \for x > 0 . \] 
    The ascending renewal measure $U(\dd x) = \LevE \int_0^\infty \Indic{H_t \in \dd x} \, \dd t$ is
    given by
    \[ U(\dd x)/\dd x = \frac{1}{\Gamma(\alpha\rho)} (1-e^{-x})^{\alpha\rho-1}
      , \for x > 0 . \]
  \item
    The descending ladder height has L\'evy density
    \[ \pi_{\hat H}(x)
      = - \frac{1}{\Gamma(-\alpha\rhohat)}
      e^{\alpha x} (e^x-1)^{-(\alpha\rhohat+1)}, \for x > 0. \]
    The descending renewal measure is given by
    \[ \hat U(\dd x)/\dd x = \frac{1}{\Gamma(\alpha\rhohat)}
      (1-e^{-x})^{\alpha\rhohat-1} e^{-(1-\alpha)x}
      , \for x > 0. \]
\end{enumerate}
\begin{proof}
The L\'evy density of $\xi$ follows from \cite[Proposition 1]{JCKuz}, and
the expressions for $\pi_H$ and $\pi_{\hat H}$ are obtained
by substituting in Section \ref{LSSs}.
The renewal measures can be verified using the Laplace transform identity
\[ \int_0^\infty e^{-\lambda x} U(\dd x) = 1/\kappa(\lambda), \for \lambda \ge 0, \]
and the corresponding identity for the descending ladder height.
\end{proof}
\end{prop}

\subsection{Wiener-Hopf factorisation for \texorpdfstring{$\alpha \in (1,2)$}{alpha in (1,2)}}
\label{s:big alpha WH}

\begin{thm}[Wiener-Hopf factorisation]
\label{big alpha WH}
\mbox{}
\begin{enumerate}[(i)]
  \item \label{big alpha WH1}%
    When $\alpha \in (1,2)$,
    the Wiener-Hopf factorisation of $\xi$ has components
    \begin{equation*}
      \kappa(\lambda) = (\alpha - 1 + \lambda)
        \frac{\Gamma(\alpha\rho + \lambda)}{\Gamma(1 + \lambda)},
      \qquad
      \hat\kappa(\lambda) = \lambda
        \frac{\Gamma(1 - \alpha\rho + \lambda)}{\Gamma(2 - \alpha + \lambda)} ,
      \for \lambda \ge 0 .
    \end{equation*}
  \item \label{big alpha WH2}%
    The ascending ladder height process can be identified as the conjugate subordinator
    (see Section \ref{define sBf}) to $\Ttrans_{\alpha - 1}\psi^*$, where
    \[ \psi^*(\lambda) = \frac{\Gamma(2-\alpha+\lambda)}{\Gamma(1-\alpha\rhohat+\lambda)} , 
      \for \lambda \ge 0 \]
    is the Laplace exponent of a Lamperti-stable process. This Lamperti-stable process
    has parameters
    \[ \bigl( q, \mathtt{a}, \, \beta, \, c , \, \LSdrift \bigr)
      = \biggl(\frac{\Gamma(2-\alpha)}{\Gamma(1-\alpha\rhohat)},  1-\alpha\rho, \, \alpha\rhohat, \,
        - \frac{1}{\Gamma(\alpha\rho-1)}, \, 0 \biggr). \]
  \item \label{big alpha WH3}
    The descending ladder process is the conjugate subordinator to a Lamperti-stable process
    with Laplace exponent
    \[ \phi^*(\lambda) = \frac{\Gamma(2-\alpha+\lambda)}{\Gamma(1-\alpha\rho+\lambda)} ,
      \for \lambda \ge 0, \]
    which has parameters
    \[ \bigl( q, \mathtt{a}, \, \beta, \, c, \, \LSdrift \bigr)
      = \biggl(\frac{\Gamma(2-\alpha)}{\Gamma(1-\alpha\rho)} ,  1-\alpha\rhohat, \, \alpha\rho, \,
        - \frac{1}{\Gamma(\alpha\rhohat-1)}, \, 0 \biggr).\]
    
\end{enumerate}

\begin{proof}
Returning to the proof of Theorem \ref{small alpha WH}\pref{small alpha WH1}, observe
that the derivation of \eqref{small alpha CE} does not depend on the value of $\alpha$. However, the factorisation for $\alpha\in(0,1]$ does not apply when $\alpha\in(1,2)$ because, for example, the expression for $\hat\kappa$ is equal to zero at $\alpha-1>0$ which contradicts the requirement that it be the Laplace exponent of a subordinator.

Now, applying the identity $x \Gamma(x) = \Gamma(x+1)$ to each denominator in that
expression, we obtain for $\theta \in \RR$
\[
  \CE(\theta)
  = (\alpha - 1 - \iu\theta)
  \frac{\Gamma(\alpha\rho - \iu\theta)}{\Gamma(1-\iu\theta)}
  \times \iu\theta
  \frac{\Gamma(1-\alpha\rho+\iu\theta)}{\Gamma(2-\alpha+\iu\theta)} .
\]
Once again, the uniqueness of the Wiener-Hopf factorisation is sufficient to
prove part \pref{big alpha WH1} once we know that $\kappa$ and $\hat \kappa$ are
Laplace exponents of subordinators, and so we now prove \pref{big alpha WH3} and
\pref{big alpha WH2}, in that order.

To prove \pref{big alpha WH3}, note that
Example 2 in \citet{KR-scale} shows that $\phi^*$ is a special Bernstein function,
conjugate to $\hat \kappa$. The fact that $\phi^*$ is the Laplace exponent of the given
Lamperti-stable process follows, as before, by substituting the parameters in
\pref{big alpha WH3} into \eqref{LSS LE}.

For \pref{big alpha WH2}, first observe that
\[ \kappa(\lambda)
  = \lambda \frac{\alpha - 1 + \lambda}{\lambda}
    \frac{\Gamma(\alpha\rho+\lambda)}{\Gamma(1+\lambda)}
  = \frac{\lambda} {\Ttrans_{\alpha - 1} \psi^*(\lambda)} . \]
It follows again from \cite[Example 2]{KR-scale} that $\psi^*$ is a special
Bernstein function, and then Proposition \ref{t:Tbeta} implies that
$\Ttrans_{\alpha-1}\psi^*$ is also a special Bernstein function, conjugate
to $\kappa$. The rest of the claim about $\psi^*$ follows as for part \pref{big alpha WH3}.
%
%
\end{proof}
\end{thm}

\begin{rem}\label{H is killed Esscher}
There is another way to view the ascending ladder height, which
is often more convenient for calculation. Applying the second part of
Proposition \ref{t:Tbeta}, we find that
\[ \kappa(\lambda) = \EsscherT_{\alpha-1}\psi(\lambda) + \psi(\alpha-1) , \]
where $\psi$ is conjugate to $\psi^*$. Hence, $H$ can be seen as the
Esscher transform of the subordinator conjugate to $\psi^*$, with additional
killing.
\end{rem}

\medskip
\begin{prop} \mbox{}
\begin{enumerate}[(i)]
  \item \label{big alpha p1}
    The process $\xi$ has L\'evy density
    \[ \LD(x) =
      \begin{cases}
        - \dfrac{1}{\Gamma(1-\alpha\rhohat)\Gamma(-\alpha\rho)} e^{-\alpha\rho x}
          \Ghg{1+\alpha\rho}{1}{1-\alpha\rhohat}{e^{-x}},
        & \text{if}\ x > 0,  \\
        - \dfrac{1}{\Gamma(1-\alpha\rho)\Gamma(-\alpha\rhohat)} e^{(1-\alpha\rho) x}
          \Ghg{1+\alpha\rhohat}{1}{1-\alpha\rho}{e^{x}},
        & \text{if}\ x < 0.
      \end{cases}
    \]
  \item \label{big alpha p2}
    The ascending ladder height has L\'evy density
    \[ \pi_H(x) = \frac{(e^x-1)^{-(\alpha\rho+1)}}{\Gamma(1-\alpha\rho)}
      \bigl( \alpha - 1 + (1-\alpha\rhohat)e^x \bigr) ,
      \for x > 0. \]
    The ascending renewal measure
    $U(\dd x) = \LevE \int_0^\infty \Indic{H_t \in \dd x}\, \dd t$ is
    given by
    \[ U(\dd x)/\dd x
      = e^{-(\alpha-1)x} \left[ \frac{\Gamma(2-\alpha)}{\Gamma(1-\alpha\rhohat)}
      + \frac{1-\alpha\rho}{\Gamma(\alpha\rho)} 
      \int_x^\infty e^{\alpha\rhohat z} (e^z - 1)^{\alpha\rho-2}\, \dd z \right] ,
      \for x > 0. \]
  \item \label{big alpha p3}
    The descending ladder height has L\'evy density
    \[ \pi_{\hat H}(x) = \frac{e^{(\alpha - 1)x}
      (e^x-1)^{-(\alpha\rhohat+1)}}{\Gamma(1-\alpha\rhohat)}
      \bigl( \alpha - 1 + (1-\alpha\rho)e^x \bigr) ,
      \for x > 0. \] 
    The descending renewal measure is given by
    \[ \hat U(\dd x)/\dd x
    = \frac{\Gamma(2-\alpha)}{\Gamma(1-\alpha\rho)}
    + \frac{1-\alpha\rhohat}{\Gamma(\alpha\rhohat)}
    \int_x^\infty e^{\alpha\rho z} (e^z - 1)^{\alpha\rhohat-2}\, \dd z ,
    \for x > 0.\] 
\end{enumerate}
\begin{proof}
As before, we will prove (i), and then (iii) and (ii) in that order.

\mbox{}(i) When $\alpha \in (1,2)$, the process $\xi$ no longer falls in the class
    of hypergeometric L\'evy processes. Therefore, although the characteristic exponent
    $\CE$ is the same as it was in Proposition \ref{small alpha p}, we can no longer rely
    on \cite{JCKuz}, and need to calculate the L\'evy density ourselves.
        
    Multiplying the jump density \eqref{jump density} of $\xiCPP$ by $c_-/\alpha$,
    we can obtain an
    expression for its L\'evy density $\LDCPP$ in terms of a $\Ghgsymb$ function. When
    we apply the relations \cite[formulas \fromto{9.131.1}{2}]{GR}, we obtain
    \[
      \LDCPP(x) = \begin{cases}
        - \dfrac{1}{\Gamma(1-\alpha\rhohat)\Gamma(-\alpha\rho)} e^{-\alpha\rho x}
        \Ghg{1+\alpha\rho}{1}{1-\alpha\rhohat}{e^{-x}} & \\
        \quad {} + \dfrac{\Gamma(\alpha+1)}{\Gamma(1+\alpha\rho)\Gamma(-\alpha\rho)}
        e^{-\alpha x} \Ghg{1+\alpha\rhohat}{\alpha+1}{1+\alpha\rhohat}{e^{-x}}, & x > 0, \\
        -\dfrac{1}{\Gamma(1-\alpha\rho)\Gamma(-\alpha\rhohat)} e^{(1-\alpha\rho)x}
        \Ghg{1+\alpha\rhohat}{1}{1-\alpha\rho}{e^x} & \\
        \quad {} - \dfrac{\Gamma(\alpha+1)}{\Gamma(\alpha\rhohat)\Gamma(1-\alpha\rhohat)}
        e^x \Ghg{1+\alpha\rho}{\alpha+1}{1+\alpha\rho}{e^x}, & x < 0 .
      \end{cases}
    \]
    Recall that $\Ghg{a}{b}{a}{z} = (1-z)^{-b}$. Then, comparing with
    \eqref{LSabs density}, the equation reads
    \[ \LDCPP(x) = \LD(x) - \LDLS(x), \quad x \ne 0 , \]
    where $\LDLS$ is the L\'evy density of $\xiLS$. The claim then follows by the
    independence of $\xiCPP$ and $\xiLS$.
      
  (iii) \hspace{0pt}
    In \cite[Example 2]{KR-scale}, the authors give the tail of the L\'evy measure
    $\Pi_{\hat H}$,
    and show that it is absolutely continuous. The density $\pi_{\hat H}$ is obtained
    by differentiation.
    
    In order to obtain the renewal measure, start with the following standard observation%
    . For $\lambda \ge 0$,
    \begin{equation}
      \int_0^\infty e^{-\lambda x} \hat U(\dd x)
      = \frac{1}{\hat\kappa(\lambda)} \\
      = \frac{\phi^*(\lambda)}{\lambda} \\
      = \int_0^\infty e^{-\lambda x} \overline{\Pi}_{\phi^*}(x)\, \dd x ,
      \label{*}
    \end{equation}
    where $\overline{\Pi}_{\phi^*}(x) = q_{\phi^*} + \Pi_{\phi^*}(x,\infty)$,
    and $q_{\phi^*}$ and $\Pi_{\phi^*}$ are, respectively, the killing rate and L\'evy
    measure of the subordinator corresponding to $\phi^*$. Comparing with section
    \ref{LSSs}, we have
    \[ q_{\phi^*} = \frac{\Gamma(2-\alpha)}{\Gamma(1-\alpha\rho)},
      \qquad
      \Pi_{\phi^*}(\dd x)/\dd x
      = - \frac{1}{\Gamma(\alpha\rhohat-1)} e^{\alpha\rho x} (e^x-1)^{\alpha\rhohat-2} , 
      \for x > 0, \]
    and substituting these back into (\ref{*}) leads immediately to the desired expression for $\hat U$.
    
  (ii) To obtain the L\'evy density, it is perhaps easier to use the representation of $H$
    as corresponding to a killed Esscher transform, noted in Remark \ref{H is killed Esscher}.
    As in part \pref{big alpha p3}, applying \cite[Example 2]{KR-scale}
    gives
    \[ \pi_{\psi}(x) = \frac{e^{(\alpha - 1)x}
      (e^x-1)^{-(\alpha\rho+1)}}{\Gamma(1-\alpha\rho)}
      \bigl( \alpha - 1 + (1-\alpha\rhohat)e^x \bigr) ,
      \for x > 0, \]
    where $\pi_{\psi}$ is the L\'evy density corresponding to
    $\psi(\lambda) = \lambda/\psi^*(\lambda)$. The effect of the Esscher transform on the
    L\'evy measure gives
    \[ \pi_H(x) = e^{-(\alpha - 1)x} \pi_\psi(x), \for x > 0, \]
    and putting everything together we obtain the required expression.
    
    Emulating the proof of \pref{big alpha p3}, we calculate
    \[
      \int_0^\infty e^{-\lambda x} U(\dd x)
      = \frac{1}{\kappa(\lambda)}
      = \frac{\psi^*(\alpha - 1 + \lambda)}{\alpha - 1 + \lambda}
      = \int_0^\infty e^{-\lambda x} e^{-(\alpha-1) x} \overline{\Pi}_{\psi^*}(x)\, \dd x ,       \]
    using similar notation to previously, and the density of $\hat U$ follows.
    \qedhere
\end{proof}
\end{prop}

\section{Proofs of main results}\label{proofs}%

In this section, we use the Wiener-Hopf factorisation of $\xi$ to prove
Theorems \ref{interval hitting} and \ref{potential int} and deduce
Corollary \ref{interval hitting prob}. We then make use of a connection
with the process conditioned to stay positive in order to prove Theorem
\ref{HP 0 before up}.

Our method for proving each theorem will be to prove a corresponding result
for the L\'evy process $\xi$,
and to relate this to the $\alpha$-stable process $X$ by means of the Lamperti transform and censoring.
In this respect, the following observation is elementary but crucial.
Let
\[ \tau_0^b = \inf \{ t > 0 : X_t \in (0,b) \}  \]
be the first time at which $X$ enters the interval $(0,b)$, where $b < 1$, and
\[ S_a^- = \inf \{ s > 0 : \xi_s < a \} \]
the first passage of $\xi$ below the negative level $a$. Notice that, if $e^a = b$, then
\[ S_a^- < \infty \text{, and } \xi_{S_a^-} \le x
  \iff \tau_0^b < \infty \text{, and } X_{\tau_0^b} \le e^x .\]
We will use this relationship several times.

Our first task is to prove Theorem \ref{interval hitting}. We split the proof
into two parts, based on the value of $\alpha$. In principle, the method
which we use for $\alpha \in (0,1]$ extends to the $\alpha \in (1,2)$ regime; however,
it requires the evaluation of an integral including the descending renewal measure.
For $\alpha \in (1,2)$ we have been unable to calculate this in closed form, and
have instead used a method based on the Laplace transform.
Conversely, the second method could be applied in the case
$\alpha\in(0,1]$; however, it is less transparent.

\medskip
\begin{proof1}
We begin by finding a related law for $\xi$.
By \cite[Proposition III.2]{BertoinLP}, for $a < 0$,
\begin{eqnarr*}
  \LevP_0(\xi_{S_a^-} \in \dd w)
  &=& \LevP_0(- \hat H_{S_{-a}^+} \in \dd w) \\
  &=& \int_{[0,-a]} \hat U(\dd z) \pi_{\hat H}(-w-z) \, \dd w .
\end{eqnarr*}
Using the expressions obtained in Section \ref{WHF} and changing variables,
\begin{eqnarr}
  \LevP_0(\xi_{S_a^-} \in \dd w)
  &=& \frac{\alpha\rhohat e^{-\alpha w} \, \dd w}
      {\Gamma(\alpha\rhohat)\Gamma(1-\alpha\rhohat)}
    \int_0^{1-e^a} t^{\alpha\rhohat-1} ( e^{-w} - 1 - e^{-w}t)^{-\alpha\rhohat - 1} \, \dd t
    \nonumber \\
  &=& \frac{\alpha\rhohat \, \dd w}
      {\Gamma(\alpha\rhohat)\Gamma(1-\alpha\rhohat)}
    e^{-\alpha\rho w} (e^{-w}-1)^{-1}
    \int_0^{\frac{1-e^a}{1-e^w}} s^{\alpha\rhohat-1}(1-s)^{-\alpha\rhohat -1} \, \dd s
    \nonumber \\
  &=& \frac{1}{\Gamma(\alpha\rhohat)\Gamma(1-\alpha\rhohat)}
    (1-e^a)^{\alpha\rhohat}
    e^{(1-\alpha\rho)w}
    (1-e^w)^{-1}
    (e^a-e^w)^{-\alpha\rhohat} \, \dd w ,
    \label{xi passage down}
\end{eqnarr}
where the last equality can be reached by \cite[formula 8.391]{GR} and the formula
$\Ghg{a}{b}{a}{z} = (1-z)^{-b}$.

Denoting by $f(a, w)$ the density on the right-hand side of \eqref{xi passage down},
the relationship between $\xi_{S_a^-}$ and $X_{\tau_0^b}$ yields that
\[
  g(b, z)
  := \stP_1(X_{\tau_0^b} \in \dd z)/\dd z
  = z^{-1} f(\log b, \log z), \for b < 1,\,\, z \in (0, b).
\]
Finally, using the scaling property we obtain
\begin{eqnarr*}
  \stP_x(X_{\tau_{-1}^1} \in \dd y)/\dd y
  &=& \frac{1}{x+1} g\biggl( \frac{2}{x+1}, \frac{y+1}{x+1} \biggr) \\
  &=& \frac{1}{y+1} f \Biggl( \log\biggl(\frac{2}{x+1}\biggr) ,
    \log\biggl(\frac{y+1}{x+1}\biggr) \Biggr) \\
  &=& \frac{\sin(\pi\alpha\rhohat)}{\pi}
    (x+1)^{\alpha\rho}
    (x-1)^{\alpha\rhohat}
    (1+y)^{-\alpha\rho}
    (1-y)^{-\alpha\rhohat}
    (x-y)^{-1}
    ,
\end{eqnarr*}
for $y \in (-1,1)$.
\qed
\end{proof1}

\begin{proof2}
We begin with the ``second factorisation identity'' \cite[Exercise 6.7]{Kyp} for
the process $\xi$, adapted to passage below a level:
\[ \int_0^\infty \int \exp(qa-\beta y)\, \Pro(a - \xi_{S_a^-}\in {\rm d}y)\, \dd a
  = \frac{\hat\kappa(q) - \hat\kappa(\beta)}{(q-\beta) \hat\kappa(q)} ,
  \for a < 0, \, q, \beta > 0. \]
A lengthy calculation, which we omit, inverts the two Laplace transforms to give
the overshoot distribution for $\xi$,
\begin{eqnarr*}
  f(a, w)
  &:=& \frac{\LevP_0(a - \xi_{S_a^-} \in \dd w)}{\dd w} \\
  &=& \frac{\sin(\pi\alpha\rhohat)}{\pi}
    e^{-(1-\alpha\rho)w}
    (1-e^{-w})^{-\alpha\rhohat} \\
  && {} \times 
    \left[ e^{(1-\alpha)a}
      (1-e^a)^{\alpha\rhohat}
      e^{-w}
      (e^{-a}-e^{-w})^{-1}
    - (\alpha\rho - 1)
      \int_0^{1-e^{a}} t^{\alpha\rhohat-1} (1-t)^{1-\alpha} \, \dd t
    \right] ,
\end{eqnarr*}
for $a < 0, w > 0$.
Essentially the same argument as in the $\alpha \in (0,1]$ case
gives the required
hitting distribution for $X$,
\begin{IEEEeqnarray}{rCll}
\frac{\stP_x(X_{\tau_{-1}^1} \in \dd y)}{\dd y}
 &=& \IEEEeqnarraymulticol{2}{l}{\frac{1}{y+1} f \Biggl( \log\biggl(\frac{2}{x+1}\biggr) ,
    \log\biggl(\frac{2}{y+1}\biggr) \Biggr)} \nonumber \\
  &=& \IEEEeqnarraymulticol{2}{l}{\frac{\sin(\pi\alpha\rhohat)}{\pi}
    (1+y)^{-\alpha\rho}
    (1-y)^{-\alpha\rhohat}} \nonumber \\
  && {} \times 
    \bigg[ &
      (y+1)
      (x-1)^{\alpha\rhohat}
      (x+1)^{\alpha\rho-1}
      (x-y)^{-1} \nonumber \\
   &&& {} - (\alpha\rho - 1) 2^{\alpha-1}
      \int_0^{\frac{x-1}{x+1}} t^{\alpha\rhohat-1} (1-t)^{1-\alpha} \, \dd t
    \bigg],
  \label{**}
\end{IEEEeqnarray}
for $x > 1$, $y \in (-1,1)$.

By the substitution $t = \frac{s-1}{s+1}$,
\begin{eqnarr*}
  2^{\alpha-1}
  \int_0^{\frac{x-1}{x+1}} t^{\alpha\rhohat-1} (1-t)^{1-\alpha} \, \dd t
  &=& 2 \int_1^x (s-1)^{\alpha\rhohat-1} (s+1)^{\alpha\rho-2} \, \dd s \\
  &=& \int_1^x (s-1)^{\alpha\rhohat-1} (s+1)^{\alpha\rho-1} \, \dd s
    - \int_1^x (s-1)^{\alpha\rhohat} (s+1)^{\alpha\rho-2} \, \dd s .
\end{eqnarr*}
Now evaluating the second term on the right hand side above via integration by parts and substituting back into (\ref{**})
yields the required law.
\qed
\end{proof2}

\begin{rem} It is worth noting that in recent work, \citet{KKP}, the law of the position of first entry of a so-called Meromorphic L\'evy process into an interval was computed as a convergent series of exponential densities by solving a pair of simultaneous non-linear equations; see Rogozin \cite{Rog71} for the original use of this method in the context of first passage of $\alpha$-stable processes when exiting a finite interval. In principle the method of solving two simultaneous non-linear equations (that is, writing the law of first entry in $(-1,1)$ from $x>1$ in terms of the law of first entry in $(-1,1)$ from $x<-1$ and vice-versa) may provide a way of proving Theorem \ref{interval hitting}. However it is unlikely that  this would  present a more convenient approach because of the complexity of the two non-linear equations involved and because of the issue of proving uniqueness of their solution.
\end{rem}

\begin{proofc1}
This will follow by integrating out Theorem \ref{interval hitting}.
First making the substitutions $z = (y+1)/2$ and $w = \frac{1-z}{1-2z/(x+1)}$,
we obtain
\begin{eqnarr*}
  \stP_x(\tau_{-1}^1 < \infty)
  &=& \frac{\sin(\pi\alpha\rhohat)}{\pi}
    (x+1)^{\alpha\rho}
    (x-1)^{\alpha\rhohat}
    \int_{-1}^1 (1+u)^{-\alpha\rho} (1-u)^{-\alpha\rhohat} (x-u)^{-1} \, \dd u \\
  &=& \frac{\sin(\pi\alpha\rhohat)}{\pi}
    (x+1)^{\alpha\rho}
    (x-1)^{\alpha\rhohat}
    2^{1-\alpha}
    \int_0^1 z^{-\alpha\rho}
      (1-z)^{-\alpha\rhohat}
      \biggl(1-\frac{2}{x+1} z\biggr)^{-1} \, \dd z \\
  &=& \frac{\sin(\pi\alpha\rhohat)}{\pi}
    \biggl( \frac{2}{x+1} \biggr)^{1-\alpha}
    \int_0^1 w^{-\alpha\rhohat}
      (1-w)^{-\alpha\rho}
      \biggl(1-\frac{2}{x+1} w\biggr)^{\alpha-1} \, \dd w \\
  &=& \frac{\Gamma(1-\alpha\rho)}{\Gamma(\alpha\rhohat)\Gamma(1-\alpha)}
    \int_0^{\frac{2}{x+1}} s^{-\alpha} (1-s)^{\alpha\rhohat-1} \, \dd s,
\end{eqnarr*}
where the last line follows by \cite[formulas 3.197.3, 8.391]{GR}. 
Finally, substituting $t = 1-s$, it follows that
\[ \stP_x(\tau_{-1}^1 = \infty)
  = \frac{\Gamma(1-\alpha\rho)}{\Gamma(\alpha\rhohat)\Gamma(1-\alpha)}
  \int_0^{\frac{x-1}{x+1}} t^{\alpha\rhohat - 1} (1-t)^{-\alpha} \, \dd t ,
\]
and this was our aim.
\qed
\end{proofc1}

\begin{proofp1}
In \citet[\S 3, Remark 3]{Por-htpr}, the author establishes,
for $s > 0$, the
hitting distribution of $[0,s]$ for a spectrally positive
$\alpha$-stable process started at $x < 0$.
In our situation, we have a spectrally negative $\alpha$-stable
process $X$, and so the dual process $\hat X$ is spectrally positive:
\begin{eqnarr*}
  \stP_x(X_{\tau_{-1}^1} \in \dd y)
  &=& \stPhat_{1-x}(\hat X_{\tau_0^2} \in 1 - \dd y) \\
  &=& f_{1-x}(1-y) \, \dd y + \gamma(1-x) \, \delta_{-1}(\dd y),
\end{eqnarr*}
using the notation from \cite{Por-htpr} in the final line.
Port gives expressions for $f_{1-x}$ and $\gamma$ which differ somewhat
from the density and atom seen in our Proposition \ref{Port limit}; our
expression
\[
  f_{1-x}(1-y)
  = \frac{\sin(\pi(\alpha-1))}{\pi}
    (x-1)^{\alpha-1}
    (1-y)^{1-\alpha}
    (x-y)^{-1}
    \Ind_{(-1,1)}(y),
\]
is obtained from Port's by evaluating an integral, and
one may compute
$\gamma(1-x)$ 
similarly.

We now prove weak convergence. For this purpose,
the identity \eqref{**} is more convenient than the final
expression in Theorem \ref{interval hitting}. Let us
denote the right-hand side of \eqref{**},
treated as the density of a measure on $[-1,1]$,
by the function $g_\rho \colon [-1,1] \to \RR$, so that
\begin{eqnarr*}
  g_\rho(y)
  &=& \frac{\sin(\pi\alpha\rhohat)}{\pi}
    (x-1)^{\alpha\rhohat}
    (x+1)^{\alpha\rho-1}
    (1+y)^{1-\alpha\rho}
    (1-y)^{-\alpha\rhohat} \\
  && {} + (1-\alpha\rho)
    \frac{\sin(\pi\alpha\rhohat)}{\pi}
    2^{\alpha-1}
    (1+y)^{-\alpha\rho}
    (1-y)^{-\alpha\rhohat} 
    \int_0^{\frac{x-1}{x+1}}
    t^{\alpha\rhohat-1} (1-t)^{1-\alpha} \, \dd t,
\end{eqnarr*}
for $y \in (-1,1)$, and we set $g_\rho(-1) = g_\rho(1) = 0$
for definiteness.

As we take the limit $\rho \to 1/\alpha$, $g_\rho(y)$
converges pointwise to $f_{1-x}(1-y)$. Furthermore,
the functions $g_\rho$ are dominated by a
function $h \colon [-1,1] \to \RR$ of the form
\[ h(y) = C (1-y)^{1-\alpha} (x-y)^{-1}
  + D (1+y)^{-1} (1-y)^{1-\alpha} , \for y \in (-1,1) \]
for some $C,D \ge 0$ depending only on $x$ and $\alpha$;
again we set $h(-1) = h(1) = 0$.

Let $z > -1$. The function $h$ is integrable on $[z,1]$,
and therefore dominated convergence yields
\[ \int_{[z,1]} g_\rho(y) \, \dd y
  \to \int_{[z,1]} f_{1-x}(1-y) \, \dd y
  = \stP_x(X_{\tau_{-1}^1} \ge z), \]
while
\[ \int_{[-1,1]} g_\rho(y) \, \dd y
  = 1 = \stP_x(X_{\tau_{-1}^1} \ge -1) , \]
and this is sufficient for weak convergence.
\qed
\end{proofp1}

\begin{proof5}
We begin by determining a killed potential for $\xi$. Let
\[ u(p, w) \, \dd w = \LevE_p\int_0^{S_0^-} \Indic{\xi_s \in \dd w}
  \, \dd s ,
  \for p, \, w > 0, \]
if this density exists.
Using an identity of Silverstein
(see \citet[Theorem VI.20]{BertoinLP}, or \citet[Theorem 6]{Sil-class}),
and the fact that the renewal measures of $\xi$ are absolutely
continuous, we find that the density $u(p,\cdot)$ does exist, and
\[
  u(p,w) = \begin{cases}
    \displaystyle\int_{p-w}^p \hat v(z) v(w+z-p) \, \dd z , & 0 < w < p, \\
    \displaystyle\int_0^p \hat v(z) v(w+z-p) \, \dd z , & w > p,
  \end{cases}
\]
where $v$ and $\hat v$ are the ascending and descending renewal densities
from Proposition \ref{small alpha p}. For $w > p$,
\begin{eqnarr*}
  u(p,w)
  &=& \frac{1}{\Gamma(\alpha\rho)\Gamma(\alpha\rhohat)}
  \int_0^p (1-e^{-z})^{\alpha\rhohat-1} 
  e^{(1-\alpha)z}
  (1-e^{p-w} e^{-z})^{\alpha\rho-1} \, \dd z \\
  &=& \frac{(1-e^{-p})^{\alpha\rhohat} (1-e^{w-p})^{\alpha\rho-1}}
    {\Gamma(\alpha\rho)\Gamma(\alpha\rhohat)}
  \int_0^1 t^{\alpha\rhohat-1}
  \bigl(1-(1-e^{-p})t\bigr)^{-\alpha}
  \biggl(1-\frac{e^{-p}-1}{e^{w-p}-1}t\biggr)^{\alpha\rho-1} \, \dd t \\
  &=& \frac{(1-e^{p-w})^{\alpha-1}}{\Gamma(\alpha\rho)\Gamma(\alpha\rhohat)}
  \biggl( \frac{1-e^{-p}}{1-e^{-w}}\biggr)^{\alpha\rhohat}
  \int_0^1 s^{\alpha\rhohat-1}
  \biggl(1- \frac{1-e^{-p}}{1-e^{-w}}s\biggr)^{-\alpha} \, \dd s,
\end{eqnarr*}
where we have used the substitution $t = 1 - \frac{e^{-z}-e^{-p}}{1-e^{-p}}$,
and then the substitution $t = s(1-q+qs)^{-1}$ with $q = \frac{e^{-p}-1}{e^{w-p}-1}$.
Finally we conclude that
\[
  u(p,w)
  = \frac{(e^{p-w}-1)^{\alpha-1}}{\Gamma(\alpha\rho)\Gamma(\alpha\rhohat)}
  \int_0^{\frac{1-e^{-w}}{1-e^{-p}}} t^{\alpha\rhohat-1} (1-t)^{-\alpha} \, \dd t,
  \for w > p.
\]
The calculation for $0 < w < p$ is very similar, and in summary we have
\[
  u(p,w) = \begin{cases}
    \dfrac{(e^{p-w}-1)^{\alpha-1}}{\Gamma(\alpha\rho)\Gamma(\alpha\rhohat)}
    \displaystyle\int_0^{\frac{1-e^{-w}}{1-e^{-p}}}
    t^{\alpha\rho-1} (1-t)^{-\alpha} \dd t
    , & 0 < w < p, \\  
    \dfrac{(1-e^{p-w})^{\alpha-1}}{\Gamma(\alpha\rho)\Gamma(\alpha\rhohat)}
    \displaystyle\int_0^{\frac{1-e^{-p}}{1-e^{-w}}}
    t^{\alpha\rhohat-1} (1-t)^{-\alpha} \dd t
    , & w > p.
  \end{cases}
\]

We can now start to calculate the killed potential for $X$. Let
\[ \bar u(b,z)\, \dd z
  = \stE_1 \int_0^{\tau_0^b} \Indic{X_t \in \dd z} \, \dd t
  , \for 0 < b < 1, \, z > b . \]
Let us recall now the censoring method and the
Lamperti transform described in Section \ref{s:cens}.
We defined $\dd A_t = \Indic{X_t > 0} \, \dd t$, denoted by
$\gamma$ the right-inverse of $A$, and
defined $Y_t = X_{\gamma(t)}\Indic{t < T_0}$ for $t \ge 0$.
Furthermore, from the Lamperti transform,
$\dd t = \exp(\alpha \xi_{S(t)}) \, \dd S(t)$, where $S$ is the Lamperti time change. As before, we write $T$ for the inverse time-change to $S$.
Finally, the measure $\stP_x$ for the stable process $X$ (and
the pssMp $Y$) corresponds under the Lamperti transform to the measure
$\LevP_{\log x}$; in particular, $\stP_1$ corresponds to $\LevP_0$,
and $\stE_1$ to $\LevE_0$.

With this in mind, we make the following calculation.
\begin{eqnarr*}
  \bar u(b,z) \, \dd z
  &=& \stE_1 \int_0^{\tau_0^b(X)}
    \Indic{X_t \in \dd z} \, \dd A_t
  = \stE_1 \int_0^{\tau_0^b(Y)}
    \Indic{Y_t \in \dd z} \, \dd t \\
  &=& \LevE_0 \int_0^{T(S_a^-)}
    \Indic{\exp(\xi_{S(t)}) \in \dd z}
    \exp(\alpha \xi_{S(t)}) \, \dd S(t) 
  = z^\alpha \LevE_0 \int_0^{S_a^-}
    \Indic{\exp(\xi_s) \in \dd z} \, \dd s , \\
  &=& z^\alpha \LevE_{-a} \int_0^{S_0^-}
    \Indic{\exp(\xi_s + a) \in \dd z} \, \dd s ,
\end{eqnarr*}
where $a = \log b$, and, for clarity, we have written $\tau_0^b(Z)$ for the
hitting time of $(0,b)$ calculated for a process $Z$.
Hence,
\[ \bar u(b,z) = z^{\alpha-1} u(\log{b^{-1}}, \log{b^{-1}z}), \for 0 < b < 1, \, z > b \]
Finally, a scaling argument yields the following. For $x \in (0,1)$ and $y > 1$,
\begin{eqnarr*}
  \stE_x \int_0^{\tau_{-1}^1} \Indic{X_t \in \dd y} \, \dd t / \dd y
  &=& (x+1)^{\alpha-1} \bar u\biggl( \frac{2}{x+1}, \, \frac{y+1}{x+1} \biggr) \\
  &=& (y+1)^{\alpha-1} u\biggl( \log \frac{x+1}{2} , \, \log \frac{y+1}{2} \biggr) \\
  &=& \begin{cases}
    \dfrac{(x-y)^{\alpha-1}}{\Gamma(\alpha\rho)\Gamma(\alpha\rhohat)}
    \dint_0^{\frac{y-1}{y+1} \frac{x+1}{x-1}}
    t^{\alpha\rho-1} (1-t)^{-\alpha} \, \dd t,
    & 1 < y < x, \\
    \dfrac{(y-x)^{\alpha-1}}{\Gamma(\alpha\rho)\Gamma(\alpha\rhohat)}
    \dint_0^{\frac{y+1}{y-1} \frac{x-1}{x+1}}
    t^{\alpha\rhohat-1} (1-t)^{-\alpha} \, \dd t,
    & y > x .
  \end{cases}
\end{eqnarr*}
The integral substitution $t = \frac{s-1}{s+1}$ gives the form in the theorem.
\qed
\end{proof5}

\medskip \noindent
We now turn to the problem of first passage upward before hitting a point.
To tackle this problem, we will use the
\emph{stable process conditioned to stay positive}.
This process has been studied by a number of authors; for a general account
of conditioning to stay positive,
see for example \citet{ChaDo05}.
If $X$ is the standard $\alpha$-stable process defined in the introduction and
\[ \tau_0^- = \inf( t \ge 0 : X_t < 0 ) \]
is the first passage time below zero, then the process conditioned to stay positive,
denoted
$\Xup$, with probability laws $(\stPup_x)_{x > 0}$, is defined as the Doob $h$-transform of
the killed process $\bigl( X_t \Indic{t < \tau_0^-}, \, t \ge 0 \bigr)$
under the invariant function
\[ h(x) = x^{\alpha\rhohat} . \]
That is, if $T$ is any a.s.\ finite stopping time, $Z$ an $\FF_T$ measurable
random variable, and $x > 0$, then
\[ \stEup_x(Z) = \stE_x \biggl[ Z \frac{h(X_T)}{h(x)} , \, T < \tau_0^- \biggr] . \]

In fact we will make use of this construction for  the dual process
$\hat X$,
with invariant function $\hat h(x) = x^{\alpha\rho}$, and accordingly we will denote the conditioned
process by $\Xhatup$ and use  $(\stPhatup_x)_{x > 0}$ for its probability laws.
It is known that the process $\Xhatup$ is a strong Markov process
which drifts to $+\infty$.

\citet{CC06} show that the process $\Xhatup$ is a pssMp, and so we can apply the
Lamperti transform to it. We will denote the L\'evy process associated to
$\Xhatup$ by $\xihatup$ with probability laws $( \LevPhatup_y)_{y >0 }$. The crucial observation here
is that $\Xhatup$ hits the point $1$ if and only if
its Lamperti transform, $\hat\xi^\uparrow$,  hits the point $0$.

We now have all the apparatus in place to begin the proof.

\begin{proof4}
For each $y \in \RR$, let $\tau_y$ be the first hitting time of
the point $y$, and
let $\tau_y^+$ and $\tau_y^-$ be the first hitting times of
the sets $(y,\infty)$ and $(-\infty,y)$, respectively.
When $\alpha \in (1,2)$, these are all a.s.\ finite
stopping times for the $\alpha$-stable process $X$ and
its dual $\hat X$. Then, when $x \in (-\infty,1)$,
\begin{eqnarr} \label{X-xiup}
  \stP_x(\tau_0 < \tau_1^+)
  = \stP_{x-1}(\tau_{-1} < \tau_0^+)
  &=& \stPhat_{1-x}(\tau_1 < \tau_0^-) \nonumber \\
  &=& \hat h(1-x)
    \stEhat_{1-x}\biggl[ \Indic{\tau_1 < \infty}
      \frac{\hat h(\hat X_{\tau_1})}{\hat h(1-x)}
      , \, \tau_1 < \tau_0^- \biggr] \nonumber \\
  &=& (1-x)^{\alpha\rho} \stPhatup_{1-x}(\tau_1 < \infty), 
\end{eqnarr}
where we have used the definition of $\stPhatup_\cdot$ at $\tau_1$.
(Note that, to unify notation, the various stopping times
refer to the canonical process for each measure.)

We now use facts coming from
\citet[Proposition II.18 and Theorem II.19]{BertoinLP}.
Provided that the potential measure
$U = \LevEhatup_0 \int_0^\infty \Indic{\xihatup \in \cdot} \, \dd t$
is absolutely continuous and there is a bounded continuous version
of its density, say $u$,
then the following holds:
\begin{equation}
  \stPhatup_{1-x}(\tau_1 < \infty)
  = \LevPhatup_{\log(1-x)}(\tau_0 < \infty)
  = C u\bigl(-\log(1-x)\bigr), \label{HP-potential}
\end{equation}
where 
$C$ is the capacity of $\{0\}$ for the process $\xihatup$.


Therefore, we have reduced our problem to that of
finding a bounded, continuous version
of the potential density of $\xihatup$
under $\LevPhatup_0$.
Provided the
renewal measures of $\xihatup$ are absolutely continuous, 
it is readily deduced from
Silverstein's identity
\cite[Theorem VI.20]{BertoinLP}
that a potential density $u$
exists and is given by
\[ u(y) = \begin{cases}
            k \int_0^\infty v(y+z) \hat v(z) \, \dd z, & y > 0 , \\
            k \int_{-y}^\infty v(y+z) \hat v(z) \, \dd z, & y < 0,
          \end{cases}
\]
where $v$ and $\hat v$ are the ascending and descending renewal densities
of the process $\xihatup$, and $k$ is the constant in the Wiener-Hopf
factorisation \eqref{the WHF} of $\xihatup$.

The work of \citet{n-tuple} gives the Wiener-Hopf factorisation of $\xihatup$,
shows that the renewal measures are absolutely continuous and
computes their densities, albeit for a different normalisation
of the $\alpha$-stable process $X$. In our normalisation, the
renewal densities are given by
\[
  v(z) = 
  \frac{1}{\Gamma(\alpha\rhohat)}
  (1-e^{-z})^{\alpha\rhohat-1} , \qquad
  \hat v(z) =
  \frac{1}{\Gamma(\alpha\rho)}
  e^{-z} (1-e^{-z})^{\alpha\rho-1},
\]
and $k = 1$.
See, for example, the computations in \cite{JCKuz}, where the
normalisation of the $\alpha$-stable process agrees with ours. 
It then follows, with similar
calculations to those in the proof of Theorem \ref{potential int},
\[ u(y) = \begin{cases}
            \frac{1}{\Gamma(\alpha\rho)\Gamma(\alpha\rhohat)}
            (1-e^{-y})^{\alpha-1}
            e^{\alpha\rho y}
            \int_0^{e^{-y}} t^{\alpha\rho-1} (1-t)^{-\alpha}\, \dd t, & y > 0, \\
            \frac{1}{\Gamma(\alpha\rho)\Gamma(\alpha\rhohat)}
            (1-e^y)^{\alpha-1}
            e^{(1-\alpha\rhohat)y}
            \int_0^{e^y} t^{\alpha\rhohat-1} (1-t)^{-\alpha} \, \dd t, & y < 0 .
          \end{cases}
\]
This $u$ is the bounded continuous density which we seek,
so by substituting into \eqref{HP-potential} and \eqref{X-xiup},
we arrive at the hitting probability
\begin{equation}\label{HP-nearly}
 \stP_x(\tau_0 < \tau_1^+)
  = \begin{cases}
      C^\prime
      x^{\alpha-1}
      \int_0^{1-x} t^{\alpha\rho-1} (1-t)^{-\alpha} \, \dd t , & 0 < x < 1, \\
      C^\prime
      (-x)^{\alpha-1}
      \int_0^{(1-x)^{-1}} t^{\alpha\rhohat-1} (1-t)^{-\alpha}\, \dd t , & x < 0,
    \end{cases}
\end{equation}
where $C^\prime = \frac{C}{\Gamma(\alpha\rho)\Gamma(\alpha\rhohat)}$.
It only remains to determine the unknown constant here, which we will
do by taking the limit  $x \upto 0$ in (\ref{HP-nearly}).
First we manipulate the second expression above, by recognising that $1 = t + (1-t)$ and
integrating by parts. For $x < 0$,
\[
\begin{split}
  \stP_x(\tau_0 < \tau_1^+)
  &= C^\prime (-x)^{\alpha-1}
  \Biggl[ \int_0^{(1-x)^{-1}} t^{\alpha\rhohat} (1-t)^{-\alpha} \, \dd t
  + \int_0^{(1-x)^{-1}} t^{\alpha\rhohat-1} (1-t)^{1-\alpha} \, \dd t \Biggr] \\
  &= C^\prime (-x)^{\alpha-1}
  \Biggl[ \frac{1}{\alpha-1} (1-x)^{\alpha\rho-1} (-x)^{1-\alpha}
  - \frac{1-\alpha\rho}{\alpha-1}
  \int_0^{(1-x)^{-1}} t^{\alpha\rhohat-1} (1-t)^{1-\alpha} \, \dd t \Biggr] \\
  &= C^\prime \frac{1}{\alpha-1} (1-x)^{\alpha\rho-1}
  - C^\prime \frac{1-\alpha\rho}{\alpha-1}
  (-x)^{\alpha-1}
  \int_0^{(1-x)^{-1}} t^{\alpha\rhohat-1} (1-t)^{1-\alpha} \, \dd t.
\end{split}
\]
Now taking $x \upto 0$, we find that $C^\prime = \alpha-1$.

Finally, we obtain the expression required by performing the integral substitution
$s = 1/(1-t)$ in \eqref{HP-nearly}.
\qed
\end{proof4}


\section*{Acknowledgements}
All three authors would like to thank the anonymous referees, whose
comments have led to great improvements in the paper.
AEK and AW gratefully acknowledge support from the Santander Research Fund.

\bibliographystyle{abbrvnat}
\bibliography{master8}

\end{document}